\documentclass[11pt]{amsart}
\usepackage{mathptmx}
\usepackage{amscd,amssymb,amsmath,amsthm,amsfonts}

\newtheorem{theorem}{Theorem}[section]
\newtheorem{prop}[theorem]{Proposition}
\newtheorem{lemma}[theorem]{Lemma}
\newtheorem{cor}[theorem]{Corollary}

\newtheorem*{T1bis}{Theorem~\ref{main:t1}bis}
\newtheorem*{C1bis}{Corollary~\ref{c1}bis}

\theoremstyle{definition}
\newtheorem{defn}{Definition}[section]
\newtheorem{remark}{Remark}[section]
\newtheorem{quest}[remark]{Question}
\newtheorem{ex}[remark]{Example}

\newtheorem*{D1bis}{Definition~\ref{norm:def}bis}
\newtheorem*{D2bis}{Definition~\ref{inv:def}bis}

\numberwithin{equation}{section}

\DeclareMathAlphabet{\matheur}{U}{eur}{m}{n}
\DeclareMathAlphabet{\matheus}{U}{eus}{m}{n}
\DeclareMathAlphabet{\matheuf}{U}{euf}{m}{n}

\newcommand{\abs}[1]{\left\lvert#1\right\rvert}

\DeclareMathOperator{\Div}{div}

\DeclareMathOperator{\dist}{dist}

\DeclareMathOperator{\diam}{diam}

\author{Asma Hassannezhad}
\address{Mathematisches Institut der Universit\"at M\"unchen, 
 Theresienstr. 39, D-80333 M\"unchen, Germany} 
 \address{Centre de recherches math\'ematiques, Universit\'e de Montr\'eal, Case postale 6128 Succursale Centre-ville Montr\'eal (QC), Canada H3C 3J7}
 \email{\tt Hassannezhad@crm.umontreal.ca}

\author{Gerasim Kokarev}
\address{Mathematisches Institut der Universit\"at M\"unchen, 
 Theresienstr. 39, D-80333 M\"unchen, Germany}
\email{\tt Gerasim.Kokarev@math.lmu.de}
 \subjclass[2010]{35P15, 58C40, 53C17, 32V20} \keywords{sub-Laplacian, eigenvalue bounds, sub-Riemannian manifold, Sasakian manifold. }

\title{Sub-Laplacian eigenvalue bounds on sub-Riemannian manifolds}

\begin{document}

\begin{abstract} 
We study eigenvalue problems for intrinsic sub-Laplacians on regular sub-Riemannian manifolds. We prove upper bounds for sub-Laplacian eigenvalues $\lambda_k$ of conformal sub-Riemannian metrics that are asymptotically sharp as $k\to +\infty$. For Sasakian manifolds with a lower Ricci curvature bound, and more generally, for contact metric manifolds conformal to such Sasakian manifolds, we obtain eigenvalue inequalities that can be viewed as versions of the classical results by Korevaar and Buser in Riemannian geometry.
\end{abstract}

\maketitle


\section{Introduction and statements of main results}
\subsection{Motivation}
Let $(M,\theta,\phi,g)$ be a compact contact metric manifold of dimension $(2\ell+1)$, possibly with boundary. Above $\theta$ is a contact form, $\phi$ is an operator field whose restriction $j(\phi)$ to the contact distribution is an almost complex structure, and $g$ is an associated Riemannian metric, see Sect. \ref{prem} for precise definitions. Denote by
$$
0=\lambda_1(g)<\lambda_2(g)\leqslant\lambda_3(g)\leqslant\ldots\leqslant\lambda_k(g)\leqslant\ldots
$$
the corresponding sub-Laplacian eigenvalues with Neumann boundary conditions. In~\cite{GK} the second-named author proved the following result.
\begin{theorem}
\label{t1}
Let $(M,\theta_0,\phi_0,g_0)$ be a compact contact metric manifold, possibly with boundary. Then there exists a constant $C$, possibly depending on the conformal class $[\theta_0]$ and $j(\phi_0)$, such that for any contact metric structure $(\theta,\phi,g)$ with $\theta=e^\varphi\theta_0$  and $j(\phi)=j(\phi_0)$ the sub-Laplacian eigenvalues $\lambda_k(g)$ satisfy the inequalities
\begin{equation}
\label{gk:eb}
\lambda_k(g)\mathit{Vol}_g(M)^{1/(\ell+1)}\leqslant C\cdot k^{1/(\ell+1)}\qquad\text{for any }k=1,2,\ldots,
\end{equation}
where the volume $\mathit{Vol}_g$ is the Riemannian volume of $g$.
\end{theorem}
Note that the Riemannian volume $\mathit{Vol}_g$ coincides up to a constant with the volume defined by the form $\theta\wedge(d\theta)^\ell$, and in particular, the left-hand side in inequality~\eqref{gk:eb} is invariant under the scaling of the contact form $\theta$. In~\cite{GK} the theorem above is stated for pseudoconvex CR manifolds, but the proof carries over directly to contact manifolds. The methods can be also applied to sub-Laplacian eigenvalue problems on rather general sub-Riemannian manifolds. They lead to eigenvalue bounds that are analogous to the celebrated results of Korevaar~\cite{Korv} on Laplace eigenvalue bounds for conformal Riemannian metrics. 

An important feature of the eigenvalue bounds in Theorem~\ref{t1} is their compatibility with the asymptotic law
$$
\lambda_k(g)\sim C_\ell\cdot(k/\mathit{Vol}_g(M))^{1/(\ell+1)}\qquad\text{as~}k\to +\infty,
$$
where $C_\ell$ is a constant that depends on the dimension of $M$ only, see~\cite{Po}. In view of this asymptotic behaviour it is extremely interesting to understand {\em up to what extent the constant in inequality~\eqref{gk:eb} can be made independent of geometry and whether one can obtain similar inequalities with explicit geometric quanities on the right-hand side.} In the present paper we obtain rather satisfactory results in these directions.

\subsection{Main results}
The principal purpose of the paper is to obtain sub-Laplacian eigenvalue bounds with an explicit dependance (or independence) of constants on geometry. We study sub-Laplacian eigenvalue problems on regular sub-Riemannian manifolds $(M,H,g)$, where $H$ is a sub-bundle of $TM$ satisfying the H\"ormander condition and $g$ is a metric on it; see Sect.~\ref{prem}-\ref{laplace} for necessary background material. In this setting the notion of an intrinsic  sub-Laplacian in general is not unique and is closely related to natural volume measures on $M$, such as the Popp measure and the Hausdorff measure. For the rest of this section we restrict our discussion to an intrinsic sub-Laplacian $(-\Delta_b)$ corresponding to the Popp measure $\mathcal{P}_g$ on a regular sub-Riemannian manifold. For contact metric manifolds this measure coincides with the volume measure of an associated Riemannian metric, and the corresponding sub-Laplacian is the standard sub-Laplacian studied in contact and CR geometry, see~\cite{DT}. In the sequel by the Hausdorff dimension of a sub-Riemannian manifold we mean the Hausdorff dimension of the Carnot-Caratheodory metric space; for regular sub-Riemannian manifolds it depends only on the algebraic properties of the distribution $H\subset TM$.

For the sake of simplicity we consider below eigenvalue problems on compact manifolds; we assume that $M$ is either closed or is a compact subdomain of a complete manifold. More general results in terms of the counting function on not necessarily compact manifolds can be found in Sect.~\ref{main}. Our  main theorem involves two {\em conformal invariants}. The first one is the function $x\mapsto\hat{\mathcal P}_x(\hat{B}_x)$ whose value at $x$ is the volume of a unit ball in the nilpotent approximation at a point $x\in M$. This invariant is unseen in Riemannian geometry, where  nilpotent approximations are isometric to the Euclidean space. On sub-Riemannian manifolds  nilpotent approximations at different points are not, in general, isometric. The second invariant is the so-called {\em conformal minimal volume} $\mathit{Min}\mathcal P[g]$, that is the infimum of volumes $\mathcal P_g(M)$ over conformal metrics satisfying certain metric covering and volume growth properties. We refer to Sect.~\ref{prem} and~\ref{main} for precise definitions and properties of these invariants. 
\begin{theorem}
\label{t2}
Let $(M,H,g_0)$ be a compact regular sub-Riemannian manifold, possibly with boundary. Then there exist positive constants $C_0$ and $C_1$ depending only on the Hausdorff dimension $Q$ of $M$ such that for any sub-Riemannian metric $g$ conformal to $g_0$ the Neumann eigenvalues $\lambda_k(g)$ of an intrinsic sub-Laplacian $(-\Delta_b)$ satisfy the inequalities
$$
\lambda_k(g)\mathcal{P}_g(M)^{2/Q}\leqslant C_0\cdot(\mathit{Min}\mathcal P[g_0])^{2/Q}+C_1\cdot(\max_M\hat{\mathcal P}_x(\hat{B}_x)^{2/Q})k^{2/Q}
$$
for any $k=1,2\ldots,$ where $\mathit{Min}{\mathcal P}[g_0]$ is the conformal minimal volume, and $\hat{\mathcal P}_x(\hat{B}_x)$ is the Popp volume of a unit ball in the nilpotent approximation at $x\in M$.
\end{theorem} 
Mention that by results in~\cite{ABB12} the function $\hat{\mathcal P}_x(\hat{B}_x)$, where $x\in M$, is continuous, and hence, bounded on a compact manifold $M$. The eigenvalue bounds in the theorem above are asymptotically sharp, in the sense that for any metric $g$  we have $\lambda_k(g)\sim C\cdot k^{2/Q}$ as $k\to+\infty$. In Sect.~\ref{laplace} we explain how this asymptotic law for the eigenvalues of an intrinsic sub-Laplacian follows from the analysis developed by M\'etivier~\cite{Me}. Following~\cite{GK}, the general idea behind the proof of Theorem~\ref{t2} is based on using  test-functions closely related to the Carnot-Caratheodory geometry of $M$. The main technical ingredient is the  decomposition theorem for general metric measure spaces with local metric covering properties, obtained by the first-named author in~\cite{Ha11}. It is built on earlier results by Grigoryan, Netrusov, and Yau~\cite{GY99,GNY} and Colbois and Maerten~\cite{CM08}. We also essentially rely on the analysis by Agrachev, Barilari, and Boscain~\cite{ABB12} concerning the asymptotic behaviour of volumes of small Carnot-Caratheodory balls and the properties of the function $\hat{\mathcal P}_x(\hat{B}_x)$. 

Theorem~\ref{t2} has a few worth mentioning consequences. First, when $M$ is a subdomain in a Carnot group $G$ equipped with a left-invariant sub-Riemannian metric $g_0$, the conformal minimal volume vanishes, and the eigenvalue bounds above reduce to the inequalities
$$
\lambda_k(g)\mathcal{P}_g(M)^{2/Q}\leqslant C\cdot \mathcal P_{g_0}(B)^{2/Q}k^{2/Q}\qquad\text{for any }k=1,2,\ldots,
$$
where $\mathcal P_{g_0}(B)$ is the volume of a unit ball in $G$, and $C$ is a constant that depends on $Q$ only. Moreover, using results in~\cite{ABB12}, we show that for corank $1$ Carnot groups the value $\mathcal P_{g_0}(B)$ is bounded by a constant that depends on the dimension only, see Lemma~\ref{volume}. The last statement yields another corollary.
\begin{cor}
\label{coro1}
Let $(M,H,g_0)$ be a compact corank $1$ sub-Riemannian manifold, possibly with boundary. Then there exist positive constants $C_0$ and $C_1$ depending only on the Hausdorff dimension $Q=\dim M+1$ such that for any sub-Riemannian metric $g$ conformal to $g_0$ the Neumann eigenvalues $\lambda_k(g)$ of an intrinsic sub-Laplacian $(-\Delta_b)$ satisfy the inequalities
$$
\lambda_k(g)\mathcal{P}_g(M)^{2/Q}\leqslant C_0\cdot(\mathit{Min}\mathcal P[g_0])^{2/Q}+C_1\cdot k^{2/Q}
$$
for any $k=1,2\ldots,$ where $\mathit{Min}{\mathcal P}[g_0]$ is the conformal minimal volume.
\end{cor}
Clearly, Corollary~\ref{coro1} covers the case when $M$ is a contact manifold and gives a sharper version of Theorem~\ref{t1}. In particular, we see that for $k\geqslant\mathit{Min}\mathcal P[g_0]$ the sub-Laplacian eigenvalues $\lambda_k(g)$ of any conformal sub-Riemannian metric $g$ satisfy the inequality
$$
\lambda_k(g)\mathcal{P}_g(M)^{2/Q}\leqslant C\cdot k^{2/Q},
$$
where $C$ is a constant that depends only on the dimension of $M$.

We proceed with specialising the considerations further to contact metric manifolds $(M,\theta,\phi,g)$. Suppose that $M$ has dimension $(2\ell+1)$ and for non-negative real numbers $t$ define the {\em volume growth function} $\alpha_{t} (g)$ as the quantity 
$$
\alpha_{t}(g)=\sup\{\mathit{Vol_g}(B(x,r))/(\omega_\ell r^{2\ell+2}): x\in M, 0<r\leqslant 1/t\},
$$ 
where $B(x,r)$ is a Carnot-Caratheodory ball, and $\omega_\ell$ is the volume of a unit ball in the Heisenberg group $\mathbb H^\ell$. As follows from known results~\cite{ABB12}, see the discussion in Sect.~\ref{prem}, the value $\alpha_t(g)$ is finite for any compact $M$. Clearly, the function $\alpha_t(g)$ is non-increasing in $t$, and, in particular, $\alpha_t(g)\leqslant\alpha_0(g)$ for any $t\geqslant 0$. In Sect.~\ref{contact} we prove the following statement.
\begin{theorem}
\label{t3}
Let $(M,\theta_0,\phi_0,g_0)$ be a compact Sasakian manifold whose Ricci curvature of a Tanaka-Webster connection is bounded below by $-1$. Then there exist positive constants $C_0$ and $C_1$ depending on the dimension of $M$ only such that for any contact metric structure $(\theta,\phi,g)$ with $\theta=e^\varphi\theta_0$ and $j(\phi)=j(\phi_0)$ the sub-Laplacian eigenvalues $\lambda_k(g)$ satisfy the inequalities
\begin{equation}
\label{new:eb}
\lambda_k(g)\mathit{Vol}_g(M)^{1/(\ell+1)}\leqslant C_0\cdot\mathit{Vol}_{g_0}(M)^{1/(\ell+1)}+C_1\cdot(\alpha_{1}(g_0)k)^{1/(\ell+1)}
\end{equation}
for any $k=1,2,\ldots$
\end{theorem}
The proof of this theorem relies on a deep work~\cite{BBG,BG,BBGM}  on curvature-dimension inequalities and volume doubling properties on Sasakian manifolds with a lower Ricci curvature bound. The statement holds also for subdomains of complete Sasakian manifolds with a lower Ricci curvature bound. Note that the value $\alpha_{t}(g_0)$ does not change under replacing $\theta_0$ with $\lambda\theta_0$ and $t$ with $t/\lambda^{1/2}$ simultaneously,  where $\lambda>0$ is a real number.  By scaling the contact form it is straightforward to see that if the assumption on the Ricci curvature in Theorem~\ref{t3} is re-placed by $\mathit{Ricci}\geqslant -t^2$, then inequality~\eqref{new:eb} takes the form
$$
\lambda_k(g)\mathit{Vol}_g(M)^{1/(\ell+1)}\leqslant C_0\cdot t^2\mathit{Vol}_{g_0}(M)^{1/(\ell+1)}+C_1\cdot(\alpha_{t}(g_0)k)^{1/(\ell+1)}.
$$
In particular, if the Ricci curvature of $g_0$ is non-negative, then tending $t\to 0+$, we obtain the following statement.
\begin{cor}
\label{coro2}
Let $(M,\theta_0,\phi_0,g_0)$ be a compact Sasakian manifold whose Ricci curvature of a Tanaka-Webster connection is non-negative. Then there exists a constant $C>0$ depending on the dimension of $M$ only such that for any contact metric structure $(\theta,\phi,g)$ with $\theta=e^\varphi\theta_0$ and $j(\phi)=j(\phi_0)$ the sub-Laplacian eigenvalues $\lambda_k(g)$ satisfy the inequalities
$$
\lambda_k(g)\mathit{Vol}_g(M)^{1/(\ell+1)}\leqslant C\cdot(\alpha_0(g_0)k)^{1/(\ell+1)}
$$
for any $k=1,2,\ldots$
\end{cor}
As another direct consequence of Theorem~\ref{t3}, we obtain the following sub-Laplacian eigenvalue bounds on Sasakian manifolds.
\begin{cor}
\label{coro3}
Let $(M,\theta,\phi,g)$ be a compact Sasakian manifold whose Ricci curvature of a Tanaka-Webster connection is bounded below by $-t^2$. Then there exist positive constants $C_0$ and $C_1$ depending on the dimension of $M$ only such that the sub-Laplacian eigenvalues $\lambda_k(g)$ satisfy the inequalities
$$
\lambda_k(g)\leqslant C_0t^2+C_1\cdot\left(\frac{\alpha_{t}(g)k}{\mathit{Vol_g(M)}}\right)^{1/(\ell+1)}
$$
for any $k=1,2,\ldots$
\end{cor}
The last two corollaries are versions of by now classical results by Korevaar~\cite{Korv} and Buser~\cite{Bu79} respectively for Laplace eigenvalues on Riemannian manifolds. In the Riemannian case the function $\alpha_{t}(g)$ does not explicitly appear, thanks to Bishop's volume comparison theorem. Due to the recent work by Agrachev and Lee~\cite{AL} and Lee and Li~\cite{LL} on volume comparison theorems on Sasakian manifolds, the function $\alpha_{t}(g)$ in the corollaries above is bounded by a constant that depends only on the dimension of $M$  when
\begin{itemize}
\item[-] the horizontal sectional curvatures of a Tanaka-Webster connection are non-negative, 
\item[-] or the manifold $M$ has dimension $3$.
\end{itemize}
In particular, in these cases we obtain direct versions of classical results for Laplace eigenvalues in Riemannian geometry.

\subsection{Organisation of the paper}
In the first section we collect the background material on sub-Riemannian geometry. We start with recalling the H\"ormander condition for a sub-bundle $H\subset TM$, the notions of the Carnot-Caratheodory distance and the nilpotent approximation. We discuss  a few examples, and show that nilpotent approximations to contact metric manifolds are isometric to the Heisenberg group with a standard left-invariant metric. We proceed with a discussion of natural volume measures on sub-Riemannian manifolds and the related results of Agrachev, Barilari, and Boscain~\cite{ABB12}.

In Sect.~\ref{laplace} we discuss eigenvalue problems for intrinsic sub-Laplacians, which have not been addressed in the literature. The eigenvalue problems for sub-Laplacians have been mostly studied on pseudoconvex CR manifolds, see the papers by Greenleaf~\cite{Green}, also~\cite[Chap.~9]{DT}, Aribi and El Soufi~\cite{AE}, Ivanov and Vassilev~ \cite{IV12,IV14} and Kokarev~\cite{GK}. One of the purposes of the present paper is to consider sub-Laplacian eigenvalue problems on rather general sub-Riemannian manifolds. We outline the necessary analysis for this setting and include a discussion on the asymptotic behaviour of eigenvalues, based on the local analysis by M\'etivier~\cite{Me}.

The next Sect.~\ref{main} and~\ref{contact} contain the main results of the paper. The former starts with a general construction of a family of invariants, called {\em conformal minimal volume}. We explain its properties and show how to choose the specific values of parameters that yield an invariant used in Theorem~\ref{t2}. The results on eigenvalue bounds here are stated in a rather general form, making them open to possible applications. As such an application, in Sect.~\ref{contact} we show how to deduce Theorem~\ref{t3} using the recent results~\cite{BBG,BG,BBGM}  on volume doubling properties on Sasakian manifolds with a lower Ricci curvature bound. We proceed with the outline of the results by Agrachev and Lee~\cite{AL} and Lee and Li~\cite{LL} on Sasakian volume comparison theorems, and use them to obtain lower bounds for the counting function, leading to Sasakian versions of classical results in Riemannian geometry. All results apply to not necessarily compact manifolds, and stated in the form whose versions even in Riemannian geometry sometimes seem to be absent in the literature.

The last Sect.~\ref{metas} is devoted to the proofs of general theorems from Sect.~\ref{main}. The paper has an appendix, where using the results in~\cite{ABB12}, we show that the volume of a unit ball in a corank $1$ Carnot group is bounded by a constant that depends on the dimension only.

\section{Preliminaries and background material}
\label{prem}
\subsection{Sub-Riemannian manifolds}
We start with recalling basic notions of sub-Riemannian geometry; for details we refer to~\cite{Be96, Mo02}. 
Let $M$ be a connected smooth manifold, possibly with boundary, and $H$ be a smooth sub-bundle of the tangent bundle $TM$, also referred to as a distribution. Suppose that the sub-bundle  $H$ satisfies the so-called {\em H\"ormander condition}: for any point $x\in M$ and any local frame $\{X_i\}$ of $H$ around $x$ the iterated Lie brackets $[X_i,X_j]$, $[X_i,[X_j,X_k]]$, $[X_i,[\ldots[X_j,X_k]\ldots]]$ at $x$ together with the vectors $X_i(x)$ span the tangent space $T_xM$. By the length of the iterated Lie bracket above we call the number of the vector fields involved. For example, the bracket $[X_i,X_j]$ has length two. For an integer $\ell\geqslant 1$ denote by $H_x^\ell$ the subspace of the tangent space $T_xM$ spanned by all iterated Lie brackets at $x$ whose length is not greater than $\ell$; the space $H^1_x$ coincides with the fiber $H_x$ of the sub-bundle $H$. Clearly, these subspaces do not depend on a choice of a frame $\{X_i\}$. The H\"ormander condition implies that for any $x\in M$ there exists an integer $r$ such that $H_x^r=T_xM$. Thus, we obtain the {\em structure filtration} of the tangent space
\begin{equation}
\label{flag}
H_x\subset H_x^2\subset\ldots\subset H_x^r=T_xM.
\end{equation}
Following~\cite{ABGR09,Be96}, the distribution $H$  is called {\em regular}  if the dimensions $n_\ell(x)=\dim H^\ell_x$, where $\ell=1,\ldots, r$, do not depend on a point $x\in M$. The minimal integer $r$ such that $H^r=TM$ is called the {\em step} of a regular distribution. Mention that in the literature there is an ambiguity concerning this notation with some authors using the term {\em equiregular} for the regular distribution, see~\cite{ABB12}.

The {\em sub-Riemannian manifold} is a triple $(M,H,g)$, where $H$ is a sub-bundle of $TM$ that satisfies the H\"ormander condition and $g$ is a smooth metric on $H$. A sub-Riemannian manifold is called {\em regular}, if the distribution $H$ is regular. Recall that an absolutely continuous path $\gamma:[0,1]\to M$ is called {\em horizontal}, if it is tangent to $H$ almost everywhere. The metric $g$ allows to measure lengths of horizontal paths by the standard formula
$$
\mathit{Length}(\gamma)=\int_0^1\abs{\dot\gamma(t)}_gdt.
$$ 
The {\em Carnot-Caratheodory distance} between the points $x$ and $y$ on a sub-Riemannian manifold is defined as
$$
d_g(x,y)=\inf\{ \mathit{Length}(\gamma): \gamma\text{ is a horizontal path joining }x\text{ and }y\},
$$
where we assume that the infimum over the empty set is equal to infinity. By the results of Chow and Rashevskij, see~\cite{Be96,Mo02}, the H\"ormander condition implies that this metric is finite and induces the original topology on $M$.

Clearly, any Riemannian manifold can be viewed as sub-Riemannian whose distribution $H$ coincides with the tangent bundle. Regular sub-Riemannian manifolds with a non-trivial horizontal distribution $H$ occur when $n=\dim M\geqslant 3$. Below we describe two major sources of their examples.
\begin{ex}[Contact manifolds]
\label{cont:def}
Let $M$ be a manifold of odd dimension $(2\ell+1)$, where $\ell\geqslant 1$. A {\em contact manifold} is a pair $(M,\theta)$, where $\theta$ is an $1$-form such that $\theta\wedge(d\theta)^\ell$ is a volume form on $M$. It is straightforward to see that on a contact manifold there is a unique vector field $\xi$, called the {\em Reeb vector field}, such that
$$
d\theta(\xi,X)=0\quad\text{and}\quad\theta(\xi)=1,
$$
for any vector field $X$. The {\em contact form} $\theta$ defines a {\em contact distribution} $H$ as a sub-bundle whose fiber is the kernel of $\theta$,
$$
H_x=\left\{X\in T_xM: \theta(X)=0\right\}.
$$
Clearly, the distribution $H$ is regular. Besides, since the form $2$-form $d\theta$ is non-degenerate on $H$, then by the relation
$$
d\theta(X,Y)=X\cdot\theta(Y)-Y\cdot\theta(X)-\theta([X,Y])
$$
we conclude that the contact distribution satisfies the H\"ormander condition. There is a special class of sub-Riemannian metrics on the contact distribution, obtained as restrictions of Riemannian metrics on $M$ associated to a contact metric structure. More precisely, a quadruple $(M,\theta,\phi,g)$, where $\theta$ is a contact form, $\phi$ is an $(1,1)$-tensor, and $g$ is a Riemannian metric, is called the {\em contact metric manifold} if
$$
\theta(X)=g(X,\xi),\quad g(X,\phi Y)=d\theta(X,Y),\quad\text{and}\quad\phi^2(X)=-X+\theta(X)\xi,
$$
for any vector fields $X$ and $Y$ on $M$.   In particular, from the last relation we see that the restriction of $\phi$ to the contact distribution, denoted by $j(\phi)$, is an almost complex structure. The relations above also imply that 
$$
g(X,Y)=g(\phi X,\phi Y),\quad \phi(\xi)=0\quad\text{and}\quad \theta\circ\phi(X)=0,
$$
for $H$-valued vector fields $X$ and $Y$, see~\cite{Blair10}. It is well-known that any contact manifold has a contact metric structure. The basic examples occur as certain hypersurfaces of K\"ahler manifolds. The special cases include CR manifolds and Sasakian manifolds, see~\cite{Blair10,DT}.
\end{ex}

\begin{ex}[Carnot groups]
\label{carnot:def}
A simply connected nilpotent Lie group $G$ with a graded Lie algebra $\mathfrak g$, that is
$$
\mathfrak g=\oplus_{i=1}^r\mathfrak g_i,\quad \mathfrak g_{i+1}=[\mathfrak g_1,\mathfrak g_i],\quad \mathfrak g_r\ne \{0\},\quad \mathfrak g_{r+1}=\{0\},
$$
is called the {\em Carnot group} of the step $r$. Any Carnot group equipped with a scalar product $\langle\cdot,\cdot\rangle$ on $\mathfrak g_1$ becomes a left-invariant regular sub-Riemannian manifold. More precisely, the distribution $H$ and the sub-Riemannian metric on it are defined by extending $(\mathfrak g_1,\langle\cdot,\cdot\rangle)$ over $G$ by left multiplications. The grading structure on the Lie algebra in particular guarantees that the distribution $H$ satisfies the H\"ormander condition. The corresponding Carnot-Caratheodory metric is also left-invariant.

Carnot groups are often viewed as sub-Riemannian manifolds equipped with the family of dilations $D_t:G\to G$. More precisely, the grading structure on the Lie algebra induces the family $\delta_t:\mathfrak g\to\mathfrak g$, where $t>0$, of Lie algebra automorphisms:
$$
\delta_t(\sum v_i)=\sum t^iv_i,\qquad v_i\in \mathfrak g_i,
$$
which preserve the grading. The dilation $D_t:G\to G$ is defined as a unique Lie group automorphism whose differential at the identity coincides with $\delta_t$. Equivalently, it can be defined as the composition $\exp\circ\delta_t\circ\exp^{-1}$, where $\exp:\mathfrak g\to G$ is the Lie exponent, which is a global diffeomorphism when $G$ is nilpotent. It is straightforward to see that the automorphisms $D_t$ preserve the left-invariant distribution $H$ and are dilations of the Carnot-Caratheodory metric, that is
$$
d(D_tx,D_ty)=td(x,y)\qquad\text{for any }x,y\in G, t>0.
$$
\end{ex}

Carnot groups are particularly important in sub-Riemannian geometry, since they play the role of tangent spaces. More precisely, they occur as the so-called {\em nilpotent approximations} of sub-Riemannian manifolds. Below we recall this notion for regular sub-Riemannian manifolds, and refer to~\cite{Be96,Mo02} for a definition in a general case and other details.

Let $(M,H,g)$ be a regular sub-Riemannian manifold of step $r\geqslant 2$. Then the structure filtration~\eqref{flag} defines a filtration of the sub-bundles $H^i\subset H^{i+1}$, where $H^r$ coincides with the tangent bundle $TM$.
For a given point $x\in M$ we consider the nilpotization at $x$ defined as the vector space
$$
\mathit{gr}(H)_x=V^1_x\oplus V^2_x\oplus\ldots\oplus V^r_x,
$$
where $V^i_x=(H^i/ H^{i-1})_x$ is the fiber of the quotient vector bundle. By the regularity hypothesis it is straightforward to see that the Lie bracket of vector fields induces a graded Lie algebra structure on $\mathit{gr}(H)_x$, that is $[V^i_x,V^j_x]\subset V_x^{i+j}$. The {\em nilpotent approximation} $\mathit{Gr}(H)_x$ of a sub-Riemannian manifold $(M,H,g)$ at a point $x$ is the Carnot group whose Lie algebra coincides with $gr(H)_x$ and is equipped with the scalar product $g_x(\cdot,\cdot)$ on the subspace $V^1_x=H_x$. In sequel by $\hat g_x$ we denote the corresponding left-invariant metric on $\mathit{Gr}(H)_x$.

An important result by Mitchell~\cite{Mi85} says that the nilpotent approximation at a point $x$, viewed as a metric space with the Carnot-Caratheodory distance, is precisely the metric tangent cone to $(M,d_g)$ at a point $x$. Mention that the nilpotent approximations at different points are generally non-isomorphic Lie groups. However, as the following example shows, if $M$ is a contact manifold, then the nilpotent approximation at every point is precisely the Heisenberg group.
\begin{ex}[Contact manifolds: continued]
\label{cont:na}
Let $V$ be a $2\ell$-dimensional vector space equipped with a symplectic form $\omega$. Recall that the {\em Heisenberg algebra} $\mathfrak h^\ell$ is a vector space $V\oplus\mathbb R$ with a product
$$
[(v,s),(w,t)]=(0,\omega(v,w)).
$$
The corresponding Carnot group is called the {\em Heisenberg group} $\mathbb H^\ell$. It is often considered as a sub-Riemannian manifold with a left-invariant metric $h$ compatible with $\omega$ on the Lie algebra. Here by a compatible metric we mean a metric that is Euclidean in a basis that takes $\omega$ to a standard symplectic form. Let $(M,\theta)$ be a contact manifold, and $H=\ker\theta$ be a contact distribution. We claim that the nilpotization $\mathit{gr}(H)_x$ is isomorphic to the Heisenberg algebra $H_x\oplus R$ equipped with the symplectic form $\omega=-d\theta$ on $H_x$. Indeed, the Lie algebra isomorphism is identical on $H_x$ and identifies $T_xM/H_x$ with $\mathbb R$ using the contact form:
$$
\mathit{gr}(H)_x=H_x\oplus(TM/H)_x\ni X\oplus T\longmapsto X\oplus\theta(T)\in H_x\oplus\mathbb R.
$$
It indeed preserves the Lie algebra structure, since $\theta([X,Y])=-d\theta(X,Y)$ for any horizontal vector fields $X$ and $Y$, see~\cite{Mo02} for details. If $(M,\theta,\phi,g)$ is a contact metric manifold, then using relations in Example~\ref{cont:def}, one can find an orthonormal basis $(X_i,-\phi X_i)$ of the space $H_x$ such that $\omega=-d\theta$ takes the standard form. In particular, we conclude that for contact metric manifolds the nilpotent approximation $\mathit{Gr}(H)_x$ equipped with a left-invariant metric $\hat g_x$ is isometric to the Heisenberg group $(\mathbb H^\ell,h)$.
\end{ex} 

\subsection{Volume measures}
Now we discuss intrinsic volume measures on sub-Rie\-mannian manifolds; these are the Hausdorff measure corresponding to the Carnot-Caratheodory metric and the Popp measure. For Riemannian manifolds both measures coincide up to a constant with the Riemannian volume measure. However, for general sub-Riemannian manifold these measures are genuinely different and lead to different intrinsic sub-Laplacians, see~\cite{Mo02, ABGR09}.

Let $(X,d)$ be a metric space. Recall that for $s>0$ the $s$-dimensional Hausdorff measure is defined as
$$
\mathcal H^s(A):=\lim_{\delta\to 0}\mathcal H^s_\delta(A)=\sup_{\delta>0}\mathcal H^s_\delta(A),
$$
where $A\subset X$, $\delta>0$, and $\mathcal H^s_\delta$ is an approximate Hausdorff measure,
\begin{equation}
\label{hs}
\mathcal H^s_\delta(A)=\inf\left\{\sum\limits_{i=1}^\infty\diam^s(A_i): A\subset\bigcup_{i=1}^{\infty} A_i, \diam (A_i)\leqslant\delta\right\}.
\end{equation}
For properties of Hausdorff measures and other details we refer to~\cite{Matt}. In sequel, we use the so-called $s$-dimensional spherical Hausdorff measure
$$
\mathcal S^s(A):=\lim_{\delta\to 0}\mathcal S^s_\delta(A)=\sup_{\delta>0}\mathcal S^s_\delta(A),
$$
where the approximate measure $\mathcal S^s_\delta$ is defined as the infimum of quantities in~\eqref{hs} with the $A_i$'s being metric balls. It is straightforward to see that
$$
\mathcal H^s(A)\leqslant\mathcal S^s(A)\leqslant 2^s\mathcal H^s(A)
$$
for any subset $A\subset X$. In particular, the Hausdorff dimension of $X$ with respect to $\mathcal H^s$-measures, that is
$$
\inf\left\{s>0:\mathcal H^s(X)=0\right\}=\sup\left\{s>0:\mathcal H^s(X)=+\infty\right\},
$$
and with respect to $\mathcal S^s$-measures coincide.

Now let $(M,H,g)$ be a sub-Riemannian manifold. Viewing $M$, as a metric space with respect to the Carnot-Caratheodory metric, we obtain Hausdorff measures on $M$. They  are intrinsic volume measures defined on an arbitrary sub-Riemannian manifold. When $M$ is regular, by the results of Mitchell~\cite{Mi85} the Hausdorff dimension of $(M,d_g)$ is given by the formula
\begin{equation}
\label{hausdorff}
Q=\sum\limits_{i=1}^r i(\dim H^i-\dim H^{i-1}),
\end{equation}
where $H^i$ are the subspaces of the structure filtration~\eqref{flag}. As is also known, the Hausdorff measures are absolutely continuous with respect to smooth measures on $M$, that is measures defined by volume forms. One of such smooth measures, the so-called {\em Popp's measure} is another intrinsic measure defined on a regular sub-Riemannian manifold. We describe it now.

Let $0=E_0\subset E_1\subset\ldots\subset E_r=E$ be a filtration of an $n$-dimensional vector space $E$. Let $\{e_i\}$ be a basis of $E$ such that $E_\ell$ is spanned by its first $n_\ell$ vectors. Since the wedge product $e_1\wedge\ldots\wedge e_n$ depends only on the equivalence classes $[e_i]\in E_{\ell_i}/E_{\ell_{i-1}}$, where $n_{\ell{i-1}}<i\leqslant n_{\ell_i}$, we obtain an isomorphism of $1$-dimensional spaces
\begin{equation}
\label{isomo}
\wedge^nE\simeq\wedge^n(\oplus_{\ell=1}^rE_\ell/E_{\ell-1}).
\end{equation}
Now let $(M,H,g)$ be a regular sub-Riemannian manifold. First, we construct the Popp volume form in a neighbourhood of every point $x\in M$. It is defined from the natural scalar product on the  nilpotization space
$$
\mathit{gr}(H)_x=H_x\oplus (H^2/H)_x\oplus\ldots\oplus (H^r/H^{r-1})_x.
$$
More precisely, due to the H\"ormander condition the map $\otimes^iH_x\to(H^i/H^{i-1})_x$ defined using sections $\{X_i\}$ of $H$ by the formula
$$
X_1\otimes\ldots\otimes X_i\longmapsto [X_1,[X_2,[\ldots,X_i]\ldots]]+H^{i-1}
$$
is surjective, and pushes forward the natural scalar product on $\otimes^i H_x$ induced by $g_x(\cdot,\cdot)$ on $H_x$. This scalar product defines, up to a sign, a volume form on $\mathit{gr}(H)_x$, that is the element of $\wedge^n\mathit{gr}(H)_x^*$. Now using isomorphism~\eqref{isomo}, we obtain the volume form on $T_xM$.

If a manifold $M$ is orientable, then this volume form $\nu_g$ can be defined globally, and induces the measure
$$
\mathcal P_g(A)=\int_A\nu_g. 
$$
For an arbitrary regular sub-Riemannian manifold $M$ the construction above defines a density, also denoted by $\nu_g$, which yields a measure by the same formula; it is called the {\em Popp measure}. For the sequel we mention that the Popp densities of conformal metrics $g$ and $\varphi\cdot g$ are related by the formula 
\begin{equation}
\label{conformal}
\nu_{\varphi\cdot g}=\varphi^{Q/2}\cdot\nu_g,
\end{equation}
where $Q$ is the Hausdorff dimension of $M$. Indeed, on each quotient $(H^i/H^{i-1})_x$ they induce scalar products that are conformal with the factor $\varphi^i$ and the comparison of their volume forms on $\mathit{gr}(H)_x$, combined with formula~\eqref{hausdorff}, yields the statement.

\subsection{Volumes of Carnot-Caratheodory balls}
We proceed with a background material on volumes of Carnot-Caratheodory balls. First, we discuss approximation results by volumes of balls in nilpotent approximations. Recall that any volume form $\omega$ on $M$ induces a natural volume form $\hat\omega_x$ on the nilpotent approximation $\mathit{Gr}(H)_x$ at a point $x\in M$. The form $\hat\omega_x$ is the left-invariant form whose value on the Lie algebra $\mathit{gr}(H)_x$ is defined by the value $\omega_x$ via the isomorphism $\wedge^n\mathit{gr}(H)_x^*\simeq\wedge^nT_x^*M$, see~\eqref{isomo}. In particular, if $\omega$ is the Popp volume form $\nu_g$, then the induced form $(\hat\nu_g)_x$ is precisely the Popp volume form $\nu_{\hat g_x}$ on the nilpotent approximation.

In the following proposition we summarize some of the results due to~\cite{ABB12} that are important for our sequel considerations. By a {\em smooth volume measure} $\mu$ below we mean a measure induced by a smooth density.
\begin{prop}
\label{agra1}
Let $(M,H,g)$ be a regular sub-Riemannian manifold, and $\mu$ be a smooth measure on $M$. Then:
\begin{itemize}
\item[(i)] for any $x\in M$ we have
$$
\mu(B(x,\varepsilon))=\varepsilon^Q\hat\mu_x(\hat B_x)+o(\varepsilon^Q)\quad\text{as }\varepsilon\to 0,
$$
where $B(x,\varepsilon)$ is a Carnot-Caratheodory ball of radius $\varepsilon$, $\hat B_x$ is a unit ball in the nilpotent approximation $\mathit{Gr}(H)_x$, and $Q$ is the Hausdorff dimension of $M$;
\item[(ii)] the measure $\mu$ is absolutely continuous with respect to the spherical Hausdorff measure $\mathcal S^Q$ and satisfies the formula
$$
\mu(A)=\frac{1}{2^Q}\int_A\hat\mu_x(\hat B_x)d\mathcal S^Q
$$
for any measurable subset $A\subset M$;
\item[(iii)] the density function $\hat\mu_x(\hat B_x)$ is always continuous, and is $C^3$-smooth if the distribution $H$ has corank one. Moreover, if $\dim M\leqslant 5$ and $\mu$ is the Popp measure $\mathcal P_g$, then it is constant, unless $H$ has corank one in dimension $5$.
\end{itemize}
\end{prop}

The proof of the statement in part~(i) above can be obtained from the distance estimates in~\cite[Sect.~7]{Be96}, which relate Carnot-Caratheodory metrics on $M$ near $x$ and on the nilpotent approximation $\mathit{Gr}(H)_x$. These distance estimates are uniform with respect to a point $x$, see also~\cite[Lemma~34]{ABB12}, and imply that the convergence $\varepsilon^{-Q}\mu(B(x,\varepsilon))\to\hat{\mu}_x(\hat{B}_x)$ is uniform in a neighborhood of $x$. Thus, we obtain the following statement.
\begin{cor}
\label{uni}
Under the hypotheses of Proposition~\ref{agra1}, for any compact subset $\Omega\subset M$ there exists $\varepsilon_0>0$ such that:
\begin{equation}
\label{in1}
\frac{1}{2}\hat{\mu}_x(\hat{B}_x)\varepsilon^Q\leqslant\mu(B(x,\varepsilon))\leqslant 2\hat{\mu}_x(\hat{B}_x)\varepsilon^Q\qquad\text{for any }0<\varepsilon\leqslant\varepsilon_0, ~x\in\Omega;
\end{equation}
\begin{equation}
\label{in2}
2^{Q-1}\varepsilon^Q\leqslant\mathcal{S}^Q(B(x,\varepsilon))\leqslant 2^{Q+1}\varepsilon^Q\qquad\text{for any }0<\varepsilon\leqslant\varepsilon_0, ~x\in\Omega.
\end{equation}
\end{cor}
\begin{proof}
Both relations in~\eqref{in1} are direct consequences of the uniform convergence 
$$
\varepsilon^{-Q}\mu(B(x,\varepsilon))\to\hat{\mu}_x(\hat{B}_x)\qquad\text{when }\varepsilon\to 0+,
$$ 
see the discussion above. Similarly, one can show that $(2\varepsilon)^{-Q}\mathcal S^Q(B(x,\varepsilon))$ converges uniformly to $1$ as $\varepsilon\to 0+$. In more detail, setting $f(x)$ to be $2^{-Q}\hat{\mu}_x(\hat{B}_x)$, by Prop.~\ref{agra1} we obtain
\begin{multline*}
(2\varepsilon)^{-Q}\mathcal{S}^Q(B(x,\varepsilon))=(2\varepsilon)^{-Q}\int_{B(x,\varepsilon)}f^{-1}(y)d\mu(y)\\=(2\varepsilon)^{-Q}f^{-1}(x)\mu(B(x,\varepsilon))+(2\varepsilon)^{-Q}\int_{B(x,\varepsilon)}(f^{-1}(y)-f^{-1}(x))d\mu(y).
\end{multline*}
As we know, the first term on the right hand-side above converges uniformly to $1$. The uniform convergence of the second term to zero follows from the uniform continuity of the density function $f$.
\end{proof}
In addition to part~(iii) of Prop.~\ref{agra1}, mention that by the result in~\cite{BG12}, the density function $\hat{\mu}_x(\hat{B}_x)$ is also $C^1$-smooth when a regular distribution $H$ has corank $2$ and step $2$. Taking the Popp measure $\mathcal P_g$ as a smooth measure in the proposition above, we obtain an intrinsic density function $(\hat{\mathcal P_g})_x(\hat B_x)$ on $M$, which plays an important role in the sequel. Using the family of dilations $D_t$ in the nilpotent approximation at $x\in M$, it is straightforward to see that
$$
(\hat{\mathcal P_g})_x(\hat B(x,tr))=t^Q(\hat{\mathcal P_g})_x(\hat B(x,r))\qquad\text{for any }t>0,
$$
where $\hat B(x,r)$ is a Carnot-Caratheodory ball in $\mathit{Gr}(H)_x$. Combining this relation with the scaling property of the Popp volume form $\nu_{\hat{g}_x}$, we conclude that the volume of a unit ball $(\hat{\mathcal P_g})_x(\hat B_x)$ does not change under a conformal change  of a metric $g$ on $H$. Recall that when $M$ is a contact metric manifold, the nilpotent approximations at different points are isometric, see Example~\ref{cont:na}. Consequently, the function $(\hat{\mathcal P_g})_x(\hat B_x)$ is constant on such manifolds. The following lemma says that, more generally, for corank one sub-Riemannian manifolds the function $(\hat{\mathcal P_g})_x(\hat B_x)$ is bounded.
\begin{lemma}
\label{volume}
Let $G$ be a Carnot group whose horizontal left-invariant distribution has corank $1$. Then there exists a constant $C=C(Q)$, depending on the Hausdorff dimension $Q$ of $G$ only, such that for any left-invariant metric $g$ on $G$ the Popp volume $\mathcal{P}_g(B_1)$ of a unit Carnot-Caratheodory ball is not greater than $C$.
\end{lemma}
The proof of Lemma~\ref{volume} is based on the explicit formula for the volume of a unit ball, obtained in~\cite{ABB12}. We collect all necessary details together with the proof of the lemma in Appendix~\ref{append}.

\section{Intrinsic sub-Laplacians and their eigenvalue problems}
\label{laplace}
\subsection{Intrinsic sub-Laplacian related to the Popp measure}
Let $(M,H,g)$ be a regular sub-Riemannian manifold of step $r$, and $\mathcal P_g$ be its Popp measure. In this section we show how the measure $\mathcal P_g$ defines an intrinsic sub-Laplacian on $M$, and discuss corresponding eigenvalue problems.

First, recall that for a smooth function $u$ on $M$ the horizontal gradient $\nabla_bu$ is defined as a unique vector field with values in the distribution $H$ such that 
$$
g(\nabla_bu,X)=du(X),\qquad\text{for any }H\text{-valued vector field }X.
$$ 
Let $\nu_g$ be a density (which locally is a smooth volume form) that defines the Popp measure $\mathcal P_g$. For any vector field $X$ on $M$ it also defines the divergence $\Div_gX$ as a function that satisfies the relation $\mathcal L_X\nu_g=\Div X\cdot\nu_g$. Following~\cite{ABGR09}, by the {\em intrinsic sub-Laplacian} $(-\Delta_b)$ we mean a second order differential operator
$$
-\Delta_bu=-\Div (\nabla_bu),\qquad\text{where~ }u\in C^\infty(M).
$$
In a local orthonormal frame $(X_i)$ of $H$ it has the following form
$$
-\Delta_b=-\sum_i X_i^2-X_0,
$$
where $X_i^2$ stand for the second Lie derivative along $X_i$, and $X_0$ is a vector field $\sum dX_j([X_i,X_j])X_i$, see~\cite{ABGR09}. Note that the vector field $X_0$ vanishes on unimodular Lie groups. Equivalently, the intrinsic sub-Laplacian can be defined as the H\"ormander operator
$$
-\Delta_b=\sum_i X_i^*X_i,
$$
where $(X_i)$ is again a local orthonormal frame of $H$, and the $X_i^*$'s are the adjoint operators with respect to the natural $L_2$-scalar product based on the Popp measure. We refer to~\cite{ABGR09} for further details and examples.

Using either local form of $\Delta_b$, by H\"ormander's theorem~\cite{Ho76} we conclude that the intrinsic sub-Laplacian is a  {\em sub-elliptic} operator. This means that for any compact subdomain $\Omega\subset M$ there exists positive constants $C$ and $\varepsilon$ such that
\begin{equation}
\label{subell}
{\vert\!\!\vert u\vert\!\!\vert}_{\varepsilon}^2\leqslant C\left(\abs{\langle\Delta_bu,u\rangle}+{{\vert\!\!\vert}{u}{\vert\!\!\vert}}^2\right)
\end{equation}
for any $C^\infty$-smooth function $u$ on $\Omega$ that is smooth up to the boundary, where ${\vert\!\!\vert\,\cdot\,\vert\!\!\vert}=\langle\cdot,\cdot\rangle^{1/2}$ and ${\vert\!\!\vert\,\cdot\,\vert\!\!\vert}_{\varepsilon}$ stand for the $L_2$-norm and the Sobolev $\varepsilon$-norm on $\Omega$ respectively. Moreover, by the results of Rotschild and Stein~\cite{RS77} the constant $\varepsilon$ can be chosen to be $(1/r)$, where $r$ is the step of the distribution $H$. It is also worth mentioning that the operator $\Delta_b$, as well as its lower order perturbations with $C^\infty$-smooth coefficients, is {\em hypoelliptic}; that is, any solution $u$ to the equation $\Delta_bu=f$, where $f$ is a $C^\infty$-smooth function, is $C^\infty$-smooth, see~\cite{Ho76}. 

When $M$ is a compact manifold, we consider $(-\Delta_b)$ as an operator defined on $C^\infty$-smooth functions that satisfy the following version of the Neumann boundary hypothesis:
\begin{equation}
\label{neumann}
\imath_{\nabla_bu}\nu_g=0\quad\text{on }\partial M,
\end{equation}
that is, the interior product of the sub-Riemannian gradient $\nabla_b$ and the Popp density $\nu_g$ vanishes. Using the divergence formula for the vector field $(v\nabla_bu)$, we obtain a sub-Riemannian version of Green's formula
$$
\int_M (\Delta_bu)vd\mathcal{P}_g+\int_M\langle\nabla_bu,\nabla_bv\rangle d\mathcal{P}_g=
\int_{\partial M}v(\imath_{\nabla_bu}\nu_g)
$$
for smooth functions $u$ and $v$ on $M$. In particular, we see that the real operator $(-\Delta_b)$ is symmetric and  admits a self-adjoint extension to an unbounded operator of $L_2(M,\mathcal{P}_g)$, see for example~\cite[Lemma~1.2.8]{Da}. Using the compact Sobolev embedding and the sub-ellipticity of $(-\Delta_b)$, it is straightforward to conclude that its resolvent is compact. Hence, the spectrum is discrete, that is, it consists of a sequence of eigenvalues
$$
0=\lambda_1(g)<\lambda_2(g)\leqslant\lambda_3(g)\leqslant\ldots\leqslant\lambda_k(g)\leqslant\ldots
$$
of finite multiplicity such that $\lambda_k(g)\to +\infty$ as $k\to +\infty$. Its {\em counting function} $N_g(\lambda)$ is defined as the number of eigenvalues counted with multiplicity that are strictly less than $\lambda$. By Green's formula above, the collection of the eigenvalues $\lambda_k(g)$ coincides also with the spectrum of the Dirichlet form $\int\abs{\nabla_bu}^2d\mathcal{P}_g$, viewed as a form on smooth functions on $M$. By hypoellipticity the eigenfunctions of $(-\Delta_b)$ are $C^\infty$-smooth, and by~\cite[Lemma~1.2.2]{Da}, it is straightforward to conclude that the self-adjoint extension of the Neumann sub-Laplacian is unique.

When a manifold $M$ is non-compact, we view the sub-Laplacian as an operator defined on $C^\infty$-smooth functions that, in addition to the Neumann boundary hypothesis, are compactly supported. This operator also admits a self-adjoint extension, and its counting function $N_g(\lambda)$ can be defined as
$$
N_g(\lambda)=\sup\left\{\dim V: \int\abs{\nabla_b u}^2d\mathcal{P}_g<\lambda\int u^2d\mathcal{P}_g\text{ for any }u\in V\backslash\{0\}\right\},
$$
where $V$ is  a subspace formed by smooth functions. We assume that the supremum over the empty set equals zero. The eigenvalues of such an operator are precisely the jump points of the counting function,
$$
\lambda_k(g)=\inf\left\{\lambda\geqslant 0: N_g(\lambda)\geqslant k\right\},
$$
see~\cite{GNY} for details.

\begin{ex}[Contact manifolds: continued]
\label{cont:laplace}
Let $(M,\theta,\phi, g)$ be a metric contact manifold of dimension $(2\ell+1)$, see Example~\ref{cont:def}. We view it as a sub-Riemannian manifold $(M,H,\left.g\right|_H)$, where $H$ is a contact distribution and $g$ is an associated Riemannian metric on $M$. It is straightforward to show that the Popp volume form $\nu_g$ coincides with the volume form of the metric $g$, and is given by the formula
$$
\nu_g=\frac{(-1)^\ell}{\ell!}\theta\wedge (d\theta)^\ell.
$$
In particular, the sub-Laplacian $(-\Delta_b)$ coincides with a standard sub-Laplacian studied in contact and CR geometry~\cite{DT}. It is related to the Laplace-Beltrami operator $\Delta_g$ for the associated Riemannian metric $g$ by the formula
\begin{equation}
\label{sub_vs_stan}
-\Delta_g=-\Delta_b+\xi^*\xi,
\end{equation}
where $\xi$ is the Reeb field on the contact structure on $M$. Note that for metric contact manifolds the Neumann boundary condition~\eqref{neumann} is equivalent to
$$
\langle\nabla_bu,\vec{n}\rangle=0\quad\text{on~}\partial M,
$$
where $\vec{n}$ is a unit outward normal to $\partial M$, and the brackets denote the scalar product in the sense of the associated metric $g$.
\end{ex}

Mention that the notation $\Delta_b$ for the sub-Laplacian is standard in contact and CR geometry. Unlike in~\cite{ABGR09}, we have chosen to  use the same notation for the intrinsic sub-Laplacian on regular sub-Riemannian manifolds to underline the fact that it coincides with the well-known operator from Example~\ref{cont:laplace}. We end this preliminary discussion with an example describing the spectrum of $\Delta_b$ for a standard metric contact structure on an odd-dimenional sphere.
\begin{ex}[Sub-Laplacian eigenvalues of spheres]
Let $S^{2\ell+1}\subset\mathbb C^{\ell+1}$ be a unit sphere equipped with the standard metric contact structure $(\theta,\phi, g)$, where $\theta=\sum(x^jdy^j-y^jdx^j)/2$ is a primitive of the K\"ahler form on $\mathbb C^{\ell+1}$, $\phi$ is the composition of the complex structure on $\mathbb C^{\ell+1}$ with the orthogonal projection to the tangent space to $S^{2\ell+1}$ at a point under consideration, and $g$ is the restriction of the Euclidean metric on $\mathbb C^{\ell+1}$. Relation~\eqref{sub_vs_stan} can be written in the form
$$
-\Delta_g=-\Delta_b-T^2,\qquad\text{where}\quad T=i\sum_{j=1}^{n+1}\left(z^j\frac{\partial}{\partial z^j}-\bar z^j\frac{\partial}{\partial\bar z^j}\right),
$$
see~\cite[p. 277]{Sta89}. Let $V^k$ be a space of homogeneous harmonic polynomials of degree $k$ on $\mathbb C^{\ell+1}$, and $V^{p,q}\subset V^k$ be a subspace formed by polynomials that are homogeneous of degree $p$ in the $z^j$'s and of degree $q$ in the $\bar z^j$'s. It is straightforward to see that $T$ is the multiplication by $i(p-q)$ on $V^{p,q}$. Recall that the subspaces $V^k$ form a complete system of eigenspaces of the Laplace operator on $S^{2\ell+1}$, where the eigenfunctions from $V^k$ correspond to an eigenvalue $k(2\ell+k)$, see~\cite{Cha}. Thus, we conclude that the subspaces $V^{p,q}$ form a complete system of eigenfunctions of the sub-Laplacian $\Delta_b$ and each eigenfunction from $V^{p,q}$ corresponds to an eigenvalue $2\ell(p+q)+4pq$. We refer to~\cite{Sta89} for the details and further references.
\end{ex}

\subsection{Intrinsic sub-Laplacian related to the Hausdorff measure}
Let $(M,H,g)$ be a closed sub-Riemannian manifold, and $\mathcal{S}^Q$ be its spherical Hausdorff measure.
In this section we consider eigenvalue problem for the form
\begin{equation}
\label{df}
u\longmapsto\int_M\abs{\nabla_bu}^2d\mathcal{S}^Q
\end{equation}
defined on $C^\infty$-smooth functions on $M$. It is straightforward to see that this form is closable, and its spectrum can be defined as the spectrum of a {\em generator} $(-\tilde\Delta_b)$, that is, a non-negative self-adjoint operator such that
$$
\int_M(-\tilde\Delta_bu)ud\mathcal{S}^Q=\int_M\abs{\nabla_bu}^2d\mathcal{S}^Q
$$ 
for any smooth function $u$, see~\cite[Theorem~4.4.2]{Da}. Though it seems to be unknown  whether in general the spectrum of $(-\tilde\Delta_b)$ is discrete, below we show that this is the case when the distribution $H$ is regular. 

Indeed, as is known the spherical Hausdorff measure is commensurable with Popp's measure $\mathcal{P}_g$, and by Prop.~\ref{agra1}, the density function $f(x)=2^{-Q}\hat{\mathcal P}_x(\hat{B}_x)$ is continuous. In particular, it is bounded away from zero and bounded above on $M$, and by inequality~\eqref{subell}, we conclude that the operator $(-\tilde\Delta_b)$ is sub-elliptic. It is then straightforward to see that its resolvent is compact, and hence, the spectrum is discrete.

When the density $f(x)$ is smooth, the operator $\tilde\Delta_b$ becomes a second order differential operator that has a form of the ``sub-Laplacian with a drift'':
$$
f\Div(f^{-1}\nabla_bu)=\Delta_bu-\langle\nabla_b\ln f,\nabla_bu\rangle,
$$
where $\Delta_b$ is the sub-Laplacian corresponding to Popp's measure. As is explained in Example~\ref{cont:na}, the density function $f$ is constant when $M$ is a contact metric manifold, and hence, in this case the sub-Laplacians $\tilde\Delta_b$ and $\Delta_b$  coincide. The same clearly holds on Carnot groups. It is worth mentioning that the density function $f$ is invariant under the conformal change of a metric $g$ on $H$, and thus, any conformal eigenvalue bounds for one of the sub-Laplacians imply automatically eigenvalue bounds for the other one via the extremal values of the function $f$. The eigenvalue bounds that we discuss in Sect.~\ref{main} have much more delicate intrinsic nature and dependance on the values of $\hat{\mathcal P}_x(\hat{B}_x)$.

\subsection{Asymptotic behaviour of eigenvalues}
To motivate our results on eigenvalue bounds, in this section we discuss asymptotic laws for eigenvalues. Throughout the rest of the section we assume that $(M,H,g)$ is a {\em closed regular} sub-Riemannian manifold. Let $\Delta_b$ be an intrinsic sub-Laplacian corresponding to Popp's measure $\mathcal{P}_g$, and $N_g(\lambda)$ be its counting function. The following result is a consequence of the analysis developed by M\'etivier~\cite{Me}.
\begin{theorem}
\label{asym}
Let $(M,H,g)$ be a closed regular sub-Riemannian manifold. Then the counting function $N_g(\lambda)$ of an intrinsic sub-Laplacian $\Delta_b$ satisfies the following asymptotic relation: $N_g(\lambda)\sim C\cdot\lambda^{Q/2}$ as $\lambda\to +\infty$, where $Q$ is the Hausdorff dimension of $M$ and $C$ is a constant that may depend on $M$ and its sub-Riemannian structure.
\end{theorem}
\begin{proof}[Outline of the proof]
Let $E(\lambda)$ be a spectral projection in the sense of the spectral theorem, see~\cite{Kato}. As is explained by M\'etivier in~\cite[Sect.~2]{Me}, it is an integral operator in $L_2(M,\mathcal{P}_g)$ with a smooth kernel $e(\lambda;x,y)$, and the counting function $N_g(\lambda)$ satisfies the relation
$$
N_g(\lambda)=\int_Me(\lambda;x,x)d\mathcal{P}_g(x),
$$
see also~\cite{GNY}. Thus, for a proof of the theorem it is sufficient to show that there exists a strictly positive continuous function $\gamma$ on $M$ such that
\begin{equation}
\label{me:aux}
\lambda^{-Q/2}e(\lambda;x,x)\longrightarrow\gamma(x)\qquad{as~}\lambda\to +\infty
\end{equation}
uniformly on $M$. For a given $x\in M$ let $(X_i)$ be a local orthonormal frame of $H$, defined on a neighbourhood $U\subset M$ of $x$.  We may assume that  the intrinsic sub-Laplacian $\Delta_b$ on $U$ has the form of the H\"ormander operator
$$
-\Delta_b=\sum X^*_iX_i.
$$
Then the analysis in~\cite{Me} shows that the values $\lambda^{-Q/2}e(\lambda;x,x)$ for a sufficiently large $\lambda$ are determined by the spectral kernel of a certain H\"ormander operator $\sum \hat{X}^*_i\hat{X}_i$ defined on compactly supported functions on the nilpotent approximation $\mathit{Gr}(H)_x$. Here $\hat X_i$ stand for left-invariant vector fields on $\mathit{Gr}(H)_x$ defined as homogeneous degree $1$ parts of the $X_i$'s at $x\in M$, see~\cite[Sect.~3]{Me} for details. In particular, the localization argument in~\cite[Sect.~4]{Me} shows that relation~\eqref{me:aux} holds on any compact subset of $U$, see~\cite[Prop.~4.6]{Me}.
\end{proof}
Mention that similar results for other classes of hypoelliptic operators have been also obtained by Menikoff and Sj\"ostrand~\cite{MS,Sj}, and Fefferman and Phong~\cite{FP80}. The asymptotic law in Theorem~\ref{asym} can be re-written in the form
$$
N_g(\lambda)\sim C(g)\cdot\lambda^{Q/2}\mathcal{P}_g(M)\qquad\text{as~}\lambda\to +\infty,
$$
where the constant $C(g)$ is invariant under the scaling of a metric $g$ on $H$. It is extremely interesting to understand how the quantity $C(g)$ depends on  a metric. As is shown by Ponge~\cite{Po}, when $M$ is a metric contact manifold, the quantity $C(g)$ depends only on the dimension of $M$.

\section{Conformal invariants and eigenvalue bounds}
\label{main}
\subsection{Conformal invariants defined by the Popp measure}
Recall that a sub-Riemannian manifold $(M,g)$ is called {\em complete}, if $M$ does not have a boundary and  the Carnot-Caratheo\-dory space $(M,d_g)$ is complete as a metric space. This hypothesis on $M$ is always assumed throughout the rest of the paper. The purpose of this section is to introduce certain conformal invariants used to study sub-Laplacian eigenvalue problems on finite volume subdomains $\Omega\subset M$, possibly coinciding with $M$. We start with the following definition.
\begin{defn}
\label{norm:def}
Given an integer $N\geqslant 1$ and a real number $\alpha\geqslant1$ a complete regular sub-Riemannian manifold $(M,H,g)$ is called {\em locally} $(N,\alpha)$-{\em normalised} if for any $0<r\leqslant 1$
\begin{itemize}
\item[(i)] each Carnot-Caratheodory ball $B(x,r)$ of radius $r$ can be covered by $N$ balls of radius $r/2$;
\item[(ii)] $\mathcal P_g(B(x,r))\leqslant \alpha\hat{\mathcal P}_x(\hat{B}_x)r^Q$ for any $x\in M$, where $Q$ is the Hausdorff dimension of $M$, and $B(x,r)$ is a Carnot-Caratheodory ball in $M$. 
\end{itemize}
If the hypotheses~$(i)$ and~$(ii)$ hold for balls of arbitrary radius $r>0$, then the metric $g$ is called {\em globally} $(N,\alpha)$-{\em normalised}. 
\end{defn}

In the sequel we refer to the hypotheses~$(i)$ and~$(ii)$ in Definition~\ref{norm:def} as the {\em covering property} and the {\em volume growth property} respectively. Using the relations
\begin{equation}
\label{relts}
\mathcal P_{\delta^{2}g}=\delta^{Q}\mathcal P_{g}\quad\text{and}\quad B_{\delta^{2}g}(x,r)=B_g(x,\delta^{-1} r),
\end{equation}
it is straightforward to see that if a metric $g$ is locally $(N,\alpha)$-normalised, then so is the metric $\delta^2\cdot g$ for any $\delta>1$. The following lemma shows that up to a scaling any metric on a compact manifold can be made normalised for an appropriate choice of the constants $N$ and $\alpha$, depending on the Hausdorff dimension of $M$ only.
\begin{lemma}
\label{norm:c}
Let $(M,H,g)$ be a complete regular sub-Riemannian manifold of Hausdorff dimension $Q$. Then for any $N\geqslant 4^{2Q+1}$, any $\alpha\geqslant 2$, and any compact subdomain $\Omega\subset M$ there exists a real number $\delta>0$ such that for the metric $\delta^2\cdot g$ the hypotheses~$(i)$ and~$(ii)$ in Definition~\ref{norm:def} hold for any Carnot-Caratheodory ball centered in $\Omega$.
\end{lemma}
\begin{proof}
It is sufficient to prove the lemma for the values $N=4^{2Q+1}$ and $\alpha=2$. By Corollary~\ref{uni} there exists $\rho>0$ such that
\begin{equation}
\label{2sides}
\frac{1}{2}\hat{\mathcal P}_x(\hat{B}_x)r^Q\leqslant \mathcal P_g(B(x,r))\leqslant 2\hat{\mathcal P}_x(\hat{B}_x)r^Q\quad\text{for any }0<r\leqslant\rho
\end{equation}
and any $x\in\Omega$. Setting $\delta=(4/\rho)$ and using relations~\eqref{relts}, we see that the metric $\delta^2\cdot g$ satisfies the relations in~\eqref{2sides} for any $0< r\leqslant 4$. In particular, the hypothesis~$(ii)$ in Definition~\ref{norm:def} holds. Now by a rather standard argument, we show that so does the hypothesis~$(i)$.

Let $B(x,r)$ be a Carnot-Caratheodory ball, where $0<r\leqslant 1$ and $x\in\Omega$, and $\{B(x_i,r/4)\}$ be a maximal family of disjoint balls centered in $B(x,r)$. Then the family $\{B(x_i,r/2)\}$ covers $B(x,r)$, and denoting by $N_x$ its cardinality, we obtain
$$
N_x\min\mathcal P_{\delta^2g}(B(x_i,r/4))\leqslant\sum\mathcal P_{\delta^2g}(B(x_i,r/4))\leqslant\mathcal P_{\delta^2g}(B(x,5r/4)).
$$
Let $x_{i_0}$ be a point at which the minimum in the left hand side is achieved. Then we have
$$
N_x\leqslant\frac{\mathcal P_{\delta^2g}(B(x,5r/4))}{\mathcal P_{\delta^2g}(B(x_{i_0},r/4))}\leqslant\frac{\mathcal P_{\delta^2g}(B(x_{i_0},4r))}{\mathcal P_{\delta^2g}(B(x_{i_0},r/4))}\leqslant 4^{2Q+1},
$$
where in the last inequality we used the fact that for the metric $\delta^2g$ relation~\eqref{2sides} holds for  any $0<r\leqslant 4$.
\end{proof}
As a consequence of Lemma~\ref{norm:c} we see that any metric on a closed manifold can be made locally $(4^{2Q+1},2)$-normalised after rescaling. The same holds for pullback metrics on covering spaces of closed manifolds. Globally normalised metrics naturally occur as left-invariant metrics on Carnot groups and as their quotients. For the convenience of references we discuss them in the examples below.
\begin{ex}[Carnot groups: continued]
\label{carnot:glo}
Let $G$ be a Carnot group equipped with a left-invariant metric $g$, see Example~\ref{carnot:def}. Using the family of dilations $D_t:G\to G$, it is straightforward to see that the Popp measure on $G$ satisfies the following dilatation property:
\begin{equation}
\label{popp:d}
\mathcal P_g(B(x,tr))=t^{Q}\mathcal P_g(B(x,r))\qquad\text{for any }t>0,
\end{equation}
where $B(x,r)$ and $B(x,tr)$ are Carnot-Caratheodory balls, and $Q$ is the Hausdorff dimension of $G$. In particular, the volume $\mathcal P_g(B(x,r))$ equals $\hat{\mathcal P}_x(\hat{B}_x)r^Q$, and the volume growth property~$(ii)$ in Definition~\ref{norm:def} is clearly satisfied with $\alpha=1$. Setting $t=2$ in relation~\eqref{popp:d}, and following the argument in the proof of Lemma~\ref{norm:c}, we see that a left-invariant metric $g$ is globally $(4^{2Q},1)$-normalised.
\end{ex}
\begin{ex}[Quotients of Carnot groups]
\label{ex:quots}
Let $G$ be a Carnot group, and $\Gamma\subset G$ be a discrete subgroup. Recall that $\Gamma$ is called {\em cocompact} (or {\em uniform}) if the quotient $G\backslash\Gamma$ is compact. The existence of a cocompact subgroup $\Gamma$ is guaranteed whenever the structure constants of the Lie algebra of $G$ are rational in some basis, see~\cite{Rag}. We assume that $G$ is endowed with a left-invariant sub-Riemannian metric $g$. The lattice $\Gamma$ acts on $G$ freely by left-multiplications, which are sub-Riemannian isometries, and the metric $g$ descends to a metric $g_*$ on the quotient $G\backslash\Gamma$. We claim that the metric $g_*$ is globally $(4^{2Q},1)$-normalised. First, for any ball $B_{g_*}(p,r)\subset G\backslash\Gamma$ there exists a fundamental domain $D$ for the action of $\Gamma$ such that
\begin{equation}
\label{nil}
\pi^{-1}(B_{g_*}(p,r))\cap D\subset B_g(x,r)\qquad\text{for some}\quad x\in\pi^{-1}(p),
\end{equation}
where $\pi:G\to G\backslash\Gamma$ is a natural projection. More precisely, for a given $x\in\pi^{-1}(p)$ the domain $D$ can be defined as the collection of $y\in G$ such that
$$
d_g(x,y)\leqslant d_g(x,\gamma\cdot y)\qquad\text{for any }\gamma\in \Gamma.
$$
In particular, we see that any point $q\in B_{g_*}(p,r)$ has a pre-image in such a set $D$. This observation together with relation~\eqref{nil} show that the volume growth property for $g_*$ is a consequence of the one for $g$. Similarly, the covering property for $g_*$ can be deduced from the covering property for $g$ and the fact that the projection $\pi:G\to G\backslash\Gamma$ does not increase the Carnot-Caratheodory distance.
\end{ex}


Now we define a family of conformal invariants based on the notion of  the Popp measure. 
\begin{defn}
\label{inv:def}
Let $H\subset TM$ be a regular distribution and $c$ be a conformal class of metrics on it. For a given integer $N\geqslant 1$ and a real number $\alpha\geqslant 1$ by the {\em conformal $(N,\alpha)$-minimal volume} of a subdomain $\Omega\subset M$, denoted by
$$
\mathit{Min}\mathcal P(\Omega,c)=\mathit{Min}\mathcal P(\Omega,c,N,\alpha),
$$ 
we call the infimum of the Popp volumes $\mathcal P_g(\Omega)$ over all locally $(N,\alpha)$-normalised metrics $g\in c$. 
\end{defn}
In the definition above we assume that the infimum over the empty set equals  infinity.  Clearly, the conformal $(N,\alpha)$-minimal volume is monotonous in $\Omega$, that is for any $\Omega_1\subset\Omega_2$ the value $\mathit{Min}\mathcal P(\Omega_1,c)$ is not greater that $\mathit{Min}\mathcal P(\Omega_2,c)$. Besides, if $\Omega$ is compact, then it is continuous with respect to an exhaustion of $\Omega$. More precisely, for any increasing sequence of subdomains $\Omega_i\subset\Omega_{i+1}$ contained in a compact $\Omega$ we have
$$
\mathit{Min}\mathcal P(\cup\Omega_i,c)=\sup\mathit{Min}\mathcal P(\Omega_i,c)=\lim\mathit{Min}\mathcal P(\Omega_i,c).
$$
The following statement clarifies the special role of globally normalised metrics.
\begin{lemma}
\label{props}
Let $c$ be a conformal class of metrics on a regular distribution $H\subset TM$. Suppose that for a given integer $N\geqslant 1$ and a real number $\alpha\geqslant 1$ there is a globally $(N,\alpha)$-normalised metric $g\in c$.
Then the conformal $(N,\alpha)$-minimal volume 
$$
\mathit{Min}\mathcal P(\Omega,c)=\mathit{Min}\mathcal P(\Omega,c,N,\alpha)
$$ 
vanishes for any subdomain $\Omega\subset M$ whose volume with respect to $g$ is finite, $\mathcal P_g(\Omega)<+\infty$.
\end{lemma}
\begin{proof}
Since a metric $g$ is globally $(N,\alpha)$-normalised, then by relations~\eqref{relts}  so is the metric $\varepsilon^2\cdot g$  for any $\varepsilon>0$. Using it as a test-metric,  we obtain
$$
\mathit{Min}\mathcal P(\Omega,c)\leqslant\mathcal P_{\varepsilon^2\cdot g}(\Omega)=\varepsilon^Q\mathcal P_g(\Omega).
$$ 
Passing to the limit as $\varepsilon\to 0+$, we prove the claim.
\end{proof}

Due to a number of open questions in the subject, at the moment of writing it is unclear what would be the best or universal choice of constants $N$ and $\alpha$ in Definitions~\ref{norm:def} and~\ref{inv:def}. That is the main reason for considering the family of invariants rather than fixing certain values of $N$ and $\alpha$. First, we would like $\mathit{Min}\mathcal P$ to be finite for a large class of complete sub-Riemannian manifolds. Second, it should reflect the intuition from Riemannian geometry that on "non-negatively curved"  objects the minimal conformal volume vanishes, see~\cite{Ha11}. On the other hand, we would like to be able to choose $N$ and $\alpha$ so that they would depend on the Hausdorff dimension of a manifold in question only.  The following lemma says that such a choice of the constants $N$ and $\alpha$ is possible in principle. 
\begin{lemma}
\label{inv:prop}
There exist an integer-valued function $N(Q)$ and a real-valued function $\alpha(Q)$, where $Q$ ranges over positive integers, such that the value
$$
\mathit{Min}\mathcal P(\Omega,c)=\mathit{Min}\mathcal P(\Omega,c,N(Q),\alpha(Q)),
$$
where $\Omega\subset M$ is a subdomain, $c$ is a conformal class of metrics on a regular distribution $H\subset TM$ and $Q$ is the Hausdorff dimension of $M$, satisfies the following properties:
\begin{itemize}
\item[(i)] it is finite, if $\Omega\subseteq M$ is compact;
\item[(ii)] it vanishes, if $\Omega$ is a finite volume subdomain in a Carnot group $G$ equipped with a left-invariant metric $g_0$, and $c$ is the conformal class of $g_0$;
\item[(ii)$'$] it vanishes, if $\Omega$ is a finite volume subdomain in a quotient $G\backslash\Gamma$ of a Carnot group $G$ by a discrete subgroup $\Gamma$ equipped with a metric $g_*$ obtained from the left-invariant metric on $G$, and $c$ is the conformal class of $g_*$;
\item[(iii)] it vanishes, if $\Omega=M$ is an odd-dimesional sphere $S^{2\ell+1}$, $\ell\geqslant 1$, and $c$ is the conformal class of a sub-Riemannian metric associated to a standard contact structure.
\end{itemize}
\end{lemma}
\begin{proof}
Set $N_1(Q)=4^{2Q+1}$ and $\alpha_1(Q)=2$. Then by Lemma~\ref{norm:c} any conformal class $c$ on a compact manifold of Hausdorff dimension $Q$ contains a locally $(N_1(Q),\alpha_1(Q))$-normalised metric, and hence, the value $\mathit{Min}\mathcal P(M,c)$ is finite. By Lemma~\ref{props}, the property~$(ii)$ is a consequence of the fact that a left invariant metric on a Carnot group $G$ is globally $(N_1(Q),\alpha_1(Q))$-normalised, see Example~\ref{carnot:glo}. Similarly, the property~$(ii)'$ follows from the existence of a globally normalised metric on a quotient $G\backslash\Gamma$, see Example~\ref{ex:quots}. Finally, note that by Corollary~\ref{uni} for any metric on a compact regular manifold there exist $N$ and $\alpha$ such that it is globally $(N,\alpha)$-normalised. Define $N_2(Q)$ and $\alpha_2(Q)$ as such values for an odd-dimensional unit sphere $S^{2\ell+1}$ with a standard contact metric structure, where $Q=2\ell+2$. Setting $N(Q)$ and $\alpha(Q)$ to be equal to the maxima of the $N_i(Q)$'s and the $\alpha_i(Q)$'s respectively, we obtain functions that satisfy the hypotheses of the lemma.
\end{proof}


As a particular case of the property~$(ii)'$ above, we see that the conformal minimal volume of any finite volume quotient $G\backslash\Gamma$ vanishes. In addition to the properties in Lemma~\ref{inv:prop}, it is harmless to assume that for Riemannian manifolds (viewed as a partial case of sub-Riemannian ones) the value $\mathit{Min}\mathcal P(M,c)$ vanishes if the conformal class $c$ contains a metric of non-negative Ricci curvature. This observation is a consequence of the Gromov relative volume comparison theorem. 

For the convenience of references we introduce the following definition.
\begin{defn}
Let $H\subset TM$ be a regular distribution and $c$ be a conformal class of metrics on it. By the {\em conformal minimal volume} of a subdomain $\Omega\subset M$ we call the quantity
$$
\mathit{Min}\mathcal P(\Omega,c)=\mathit{Min}\mathcal P(\Omega,c,N(Q),\alpha(Q)),
$$
where $Q$ is the Hausdorff dimension of $M$, and $N(Q)$ and $\alpha(Q)$ are functions that satisfy the conclusions~$(i)-(iii)$ of Lemma~\ref{inv:prop}.
\end{defn}

As a direct consequence of the property~$(iii)$ in Lemma~\ref{inv:prop} we also have the following statement.
\begin{cor}
\label{cora:h}
Let $\mathbb H^\ell$ be a Heisenberg group equipped with a standard left-invariant sub-Riemannian metric $g_0$. Then the conformal minimal volume $\mathit{Min}\mathcal P(\mathbb H^\ell,[g_0])$ vanishes.
\end{cor}
\begin{proof}
Recall that a standard left-invariant sub-Riemannian metric $g_0$ on $\mathbb H^\ell$, described in Example~\ref{cont:na}, can be viewed as a metric associated to a standard contact metric structure on $\mathbb H^\ell$. Moreover, as is known~\cite{DT}, there is a CR diffeomorphism that maps the punctured sphere $S^{2\ell+1}\backslash\{\mathit{pt}\}$ onto  the Heisenberg group $\mathbb H^\ell$. In particular, any sub-Riemannian metric on $S^{2\ell+1}$ conformal to the metric associated to a standard contact structure is pushed forward  to a sub-Riemannian metric on $\mathbb H^\ell$ conformal to $g_0$. Thus, we see that
$$
\mathit{Min}\mathcal P(\mathbb H^\ell,[g_0])\leqslant \mathit{Min}\mathcal P(\mathbb S^{2\ell+1},[g_0]),
$$
and the statement follows from Lemma~\ref{inv:prop}.
\end{proof}

The last statement leads to the following more general question.
\begin{quest}
\label{que1}
Do there exist functions $N(Q)$ and $\alpha(Q)$ of a positive integer parameter $Q$ such that for any Carnot group of Hausdorff dimension $Q$ equipped with a left-invariant sub-Riemannian metric $g$ the conformal minimal volume
$$
\mathit{Min}\mathcal P(G,[g])=\mathit{Min}\mathcal P(G,[g],N(Q),\alpha(Q))
$$
vanishes?
\end{quest}

\subsection{Eigenvalue bounds}
Now  we state our first main result. It gives a {\em parametric} lower bound for the counting function of the sub-Laplacian $(-\Delta_b)$, which applies to rather general sub-Riemannian manifolds. In the sequel by a subdomain $\Omega\subset M$ we mean a subdomain with a smooth boundary. By $\lfloor\lambda\rfloor$ we denote the floor function of $\lambda\in\mathbb R$, the greatest integer that is at most $\lambda$. The proofs of the theorems below appear in Sect.~\ref{metas}.
\begin{theorem}
\label{main:t1}
Let $(M,H,g_0)$ be a complete regular sub-Riemannian manifold. Then for any integer $N\geqslant 1$ and a real number $\alpha\geqslant 1$ there exist positive constants $C_0=C_0(N)$ and $C_1=C_1(N,Q)$ depending only on $N$ and  the Hausdorff dimension $Q$ of $M$ such that for any subdomain $\Omega\subset M$ and any sub-Riemannian metric $g$ that is conformal to $g_0$ on $\Omega$ and has finite volume $\mathcal P_g(\Omega)<+\infty$ the counting function $N_g(\Omega,\lambda)$ for the Neumann sub-Laplacian $(-\Delta_b)$ on $\Omega$ satisfies the inequality
$$
N_g(\Omega,\lambda)\geqslant \lfloor C_1(\alpha\max_\Omega\hat{\mathcal P}_x(\hat{B}_x))^{-1}(\mathcal P_g(\Omega)^{2/Q}\lambda-C_0(\mathit{Min}\mathcal P(\Omega,[g_0]))^{2/Q})^{Q/2}\rfloor
$$
for any $\lambda\geqslant C_0(\mathit{Min}\mathcal P(\Omega,[g_0])/\mathcal P_g(\Omega))^{2/Q}$, where  $\hat{\mathcal P}_x(\hat{B}_x)$ is the volume of a unit ball in  the nilpotent approximation at $x\in M$, and
$$
\mathit{Min}\mathcal P(\Omega,[g_0])=\mathit{Min}\mathcal P(\Omega,[g_0],N,\alpha)  
$$
is the conformal $(N,\alpha)$-minimal volume of $\Omega\subset M$. Besides, if a subdomain $\Omega\subset M$ is compact, then the floor function is unnecessary in the estimate above.
\end{theorem}
Choosing $N(Q)$ and $\alpha(Q)$ as in Lemma~\ref{inv:prop}, we obtain a lower bound for the counting function with the constants $C_0$ and $(C_1/\alpha)$ depending only on the Hausdorff dimension $Q$, and with the conformal minimal volume satisfying the properties~$(i)$-$(iii)$ in Lemma~\ref{inv:prop}. 
\begin{cor}
\label{c1}
Let $(M,H,g_0)$ be a complete regular sub-Riemannian manifold. Then there exist positive constants $\bar C_0=\bar C_0(Q)$ and $\bar C_1=\bar C_1(Q)$ depending only on the Hausdorff dimension $Q$ of $M$ such that for any subdomain $\Omega\subset M$ and any sub-Riemannian metric $g$ that is conformal to $g_0$ on $\Omega$ and has finite volume $\mathcal P_g(\Omega)<+\infty$ the counting function $N_g(\Omega,\lambda)$ for the Neumann sub-Laplacian $(-\Delta_b)$ on $\Omega$ satisfies the inequality
$$
N_g(\Omega,\lambda)\geqslant \lfloor\bar C_1(\max_\Omega\hat{\mathcal P}_x(\hat{B}_x))^{-1}(\mathcal P_g(\Omega)^{2/Q}\lambda-\bar C_0(\mathit{Min}\mathcal P(\Omega,[g_0]))^{2/Q})^{Q/2}\rfloor
$$
for any $\lambda\geqslant \bar C_0(\mathit{Min}\mathcal P(\Omega,[g_0])/\mathcal P_g(\Omega))^{2/Q}$, where  $\hat{\mathcal P}_x(\hat{B}_x)$ is the volume of a unit ball in  the nilpotent approximation at $x\in M$, and $\mathit{Min}\mathcal P(\Omega,[g_0])$ is the conformal minimal volume of $\Omega\subset M$. Besides, if a subdomain $\Omega\subset M$ is compact, then the floor function is unnecessary in the estimate above.
\end{cor}

In particular, if $\Omega$ is a {\em compact} subdomain in a Carnot group $G$ equipped with a left-invariant metric $g_0$, or in a quotient $G\backslash\Gamma$ by a discrete subgroup, then the conformal minimal volume vanishes, and we have
\begin{equation}
\label{est:carnot}
N_g(\Omega,\lambda)\geqslant C\cdot(\mathcal P_{g_0}(B))^{-1}\mathcal P_g(\Omega)\lambda^{Q/2}\qquad\text{for any }\lambda\geqslant 0,
\end{equation}
where $\mathcal P_{g_0}(B)$ is the volume of a unit ball in $G$, and  the constant $C$ depends only on the Hausdorff dimension $Q$ of $G$.  More generally, if for an appropriate choice of $N(Q)$ and $\alpha(Q)$ the conformal minimal volume vanishes on some Carnot group $G$, then inequality~\eqref{est:carnot} holds for {\em any} (not necessarily compact) subdomain $\Omega$ in $G$ or in a quotient $G\backslash\Gamma$. For example, this hypothesis holds for the Heisenberg group $\mathbb H^\ell$, see Corollary~\ref{cora:h}. Independently whether it holds or not in general (see Question~\ref{que1}) we have the following result that gives eigenvalue bounds of the form~\eqref{est:carnot}.
\begin{theorem}
\label{main:t2}
Let $(M,H,g_0)$ be a complete regular sub-Riemannian manifold that is globally $(N,\alpha)$-normalised for some integer $N\geqslant 1$ and real number $\alpha\geqslant 1$. Then there exists a constant $C=C(N,Q)>0$ depending on $N$ and  the Hausdorff dimension $Q$ of $M$ such that for any subdomain $\Omega\subset M$ and any sub-Riemannian metric $g$ that is conformal to $g_0$ on $\Omega$ and has finite volume $\mathcal P_g(\Omega)<+\infty$ the counting function $N_g(\Omega,\lambda)$ for the Neumann sub-Laplacian $(-\Delta_b)$ on $\Omega$ satisfies the inequality
$$
N_g(\Omega,\lambda)\geqslant \lfloor C(\alpha\max_\Omega\hat{\mathcal P}_x(\hat{B}_x))^{-1}\mathcal P_g(\Omega)\lambda^{Q/2}\rfloor
$$
for any $\lambda\geqslant 0$, where  $\hat{\mathcal P}_x(\hat{B}_x)$ is the volume of a unit ball in  the nilpotent approximation at $x\in M$. Besides, if a subdomain $\Omega\subset M$ is compact, then the floor function is unnecessary in the estimate above.
\end{theorem}
Recall that, by Example~\ref{carnot:glo}, any left-invariant metric $g_0$ on a Carnot group $G$ is globally $(4^{2Q},1)$-normalised. Thus, as a consequence of the last theorem, we see that inequality~\eqref{est:carnot} holds for an {\em arbitrary} subdomain $\Omega\subset G$ and a metric $g$ conformal to $g_0$ such that $\mathcal P_g(\Omega)<+\infty$. 

We proceed with explaining how Theorem~\ref{t2} follows from Corollary~\ref{c1}.
\begin{proof}[Proof of Theorem~\ref{t2}]
Recall that the eigenvalues are related to the counting function by the following formula
$$
\lambda_k(g)=\inf\{\lambda\geqslant 0: N_g(\lambda)\geqslant k\}.
$$
For a given positive integer $k$ take as $\lambda$ the value
$$
(\bar C_0\cdot(\mathit{Min}\mathcal P[g_0])^{2/Q}+\bar C_1^{-2/Q}\cdot(\max_M\hat{\mathcal P}_x(\hat{B}_x)^{2/Q})k^{2/Q})\mathcal P_g(M)^{-2/Q},
$$
where $\bar C_0$ and $\bar C_1$ are constants from Corollary~\ref{c1}. Clearly, for this choice of  $\lambda$ the hypotheses of Corollary~\ref{c1} are satisfied, and by the estimate for the counting function, we obtain that $N_g(\lambda)\geqslant k$, and conclude that $\lambda_k(g)\leqslant\lambda$.
\end{proof}

In view of obtaining eigenvalue bounds less dependent on the geometry of $M$, it is interesting to understand when the volume of a unit ball $\hat{\mathcal P}_x(\hat{B}_x)$ can be bounded by a quantity that depends on the dimension of $M$ only. Up to our knowledge, the following basic question is open.
\begin{quest}
Does there exist a real-valued function $C(Q)$ of a positive integer parameter $Q$ such that for any Carnot group of the Hausdorff dimension $Q$ equipped with a left-invariant metric $g$ the volume of a unit ball $\mathcal P_g(B)$ is at most $C(Q)$?
\end{quest} 
By Lemma~\ref{volume} the answer to this question is positive for corank $1$ Carnot groups. As a consequence, for corank $1$ sub-Riemannian manifolds the quantity $\max\hat{\mathcal P}_x(\hat{B}_x)$ can be dispensed with in Theorems~\ref{main:t1} and~\ref{main:t2}. Further, as a consequence of Corollary~\ref{c1} (or Theorem~\ref{t2}), for compact subdomains $\Omega$ in corank $1$ sub-Riemannian manifolds $M$ we obtain eigenvalue bounds
\begin{equation}
\label{ei:popp}
\lambda_k(g)\mathcal{P}_g(\Omega)^{2/Q}\leqslant C_0\cdot(\mathit{Min}\mathcal P(\Omega,[g_0]))^{2/Q}+C_1\cdot k^{2/Q},
\end{equation} 
stated in the introduction as Corollary~\ref{coro1}.
\subsection{Eigenvalue bounds for the Hausdorff sub-Laplacian}
As was touched upon above, it is important to know whether the eigenvalue bounds in~\eqref{ei:popp} hold for higher corank sub-Riemannian manifolds. Recall that when $M$ is a contact metric manifold the intrinsic sub-Laplacians $\Delta_b$ and $\tilde\Delta_b$ corresponding to the Popp and the Hausdorff measures respectively coincide, and inequalities~\eqref{ei:popp} can be also viewed as upper bounds for the eigenvalues $\tilde\lambda_k(g)$ of the latter sub-Laplacian. The purpose of the remaining part of the section is to show that {\em these upper bounds hold for $\tilde\lambda_k(g)$ on rather arbitrary compact sub-Riemannian manifolds}. The definitions and results below are direct analogues of the ones discussed earlier. The principal difference is that the quantity $\hat{\mathcal P}_x(\hat{B}_x)$ is already taken into the account in the definition of $\tilde\Delta_b$ and  does not participate in the statements of the results. 

Following the line of exposition above, we first define the notion of $(N,\alpha)$-normalised metric. In the sequel by $\mathcal S_g$ we denote the $Q$-dimensional spherical Hausdorff measure $\mathcal S^Q$ on the Carnot-Caratheodory space $(M,d_g)$.
\begin{D1bis}
Given an integer $N\geqslant 1$ and a real number $\alpha\geqslant1$ a complete regular sub-Riemannian manifold $(M,H,g)$ is called {\em locally} $(N,\alpha)$-{\em normalised} if for any $0<r\leqslant 1$
\begin{itemize}
\item[(i)] each Carnot-Caratheodory ball $B(x,r)$ of radius $r$ can be covered by $N$ balls of radius $r/2$;
\item[(ii)] $\mathcal S_g(B(x,r))\leqslant \alpha (2r)^Q$ for any $x\in M$, where $Q$ is the Hausdorff dimension of $M$, and $B(x,r)$ is a Carnot-Caratheodory ball in $M$. 
\end{itemize}
If the hypotheses~$(i)$ and~$(ii)$ hold for balls of arbitrary radius $r>0$, then the metric $g$ is called {\em globally} $(N,\alpha)$-{\em normalised}. 
\end{D1bis}
We proceed with a definition of conformal $(N,\alpha)$-minimal volume with respect to the Hausdorff measure.
\begin{D2bis}
Let $H\subset TM$ be a regular distribution and $c$ be a conformal class of metrics on it. For a given integer $N\geqslant 1$ and a real number $\alpha\geqslant 1$ by the {\em conformal $(N,\alpha)$-minimal Hausdorff volume} of a subdomain $\Omega\subset M$, denoted by
$$
\mathit{Min}\mathcal S(\Omega,c)=\mathit{Min}\mathcal S(\Omega,c,N,\alpha),
$$ 
we call the infimum of the Hausdorff volumes $\mathcal S_g(\Omega)$ over all locally $(N,\alpha)$-normalised metrics $g\in c$. 
\end{D2bis}

It is straightforward to see that the main properties of the conformal $(N,\alpha)$-minimal volume continue to hold in this new setting. In particular, so does the version of Lemma~\ref{inv:prop}, stating that there exist an integer-valued function $N(Q)$ and a real-valued function $\alpha(Q)$ such that the value
\begin{equation}
\label{inv:haus}
\mathit{Min}\mathcal S(\Omega,c)=\mathit{Min}\mathcal S(\Omega,c,N(Q),\alpha(Q))
\end{equation}
satisfies the same natural finiteness and vanishing properties. In fact, the functions $N(Q)$ and $\alpha(Q)$ can be chosen to be the same as in the proof of Lemma~\ref{inv:prop}. Mention also that the invariants $\mathit{Min}\mathcal P(\Omega,c)$ and $\mathit{Min}\mathcal S(\Omega,c)$ vanish or not simultaneously for compact subdomains $\Omega\subset M$. The relationship between them for more general subdomains can be described via the properties of the density function $f(x)=2^{-Q}\hat{\mathcal P}_x(\hat{B}_x)$.

We end with stating lower bounds for the counting function $\tilde N_g(\lambda)$ of the Hausdorff sub-Laplacian $(-\tilde\Delta_b)$. Note  that these results are independent of the corresponding bounds for $N_g(\lambda)$ of the Popp sub-Laplacian, and in general, can not be derived from Theorem~\ref{main:t1}. Throughout the rest of the section we assume that $M$ is a closed regular sub-Riemannian manifold. The results continue to hold for subdomains $\Omega$ in complete regular sub-Riemannian manifolds, if the Neumann problem is well-defined. The latter, for example, occurs when the density function $f(x)=2^{-Q}\hat{\mathcal P}_x(\hat{B}_x)$ is $C^1$-smooth on $\Omega$.
\begin{T1bis}
Let $(M,H,g_0)$ be a closed regular sub-Riemannian manifold. Then for any integer $N\geqslant 1$ and a real number $\alpha\geqslant 1$ there exist positive constants $C_0=C_0(N)$ and $C_1=C_1(N,Q)$ depending only on $N$ and  the Hausdorff dimension $Q$ of $M$ such that for any sub-Riemannian metric $g$ conformal to $g_0$ the counting function $\tilde N_g(\lambda)$ for the sub-Laplacian $(-\tilde\Delta_b)$  satisfies the inequality
$$
\tilde N_g(\lambda)\geqslant C_1\alpha^{-1}(\mathcal S_g(M)^{2/Q}\lambda-C_0(\mathit{Min}\mathcal S(M,[g_0]))^{2/Q})^{Q/2}
$$
for any $\lambda\geqslant C_0(\mathit{Min}\mathcal S(M,[g_0])/\mathcal S_g(\Omega))^{2/Q}$, where  
$$
\mathit{Min}\mathcal S(M,[g_0])=\mathit{Min}\mathcal S(M,[g_0],N,\alpha)  
$$
is the conformal $(N,\alpha)$-minimal Hausdorff volume of $M$. 
\end{T1bis}
The proof of the theorem follows the same argument as the proof of Theorem~\ref{main:t1} and is discussed at the end of Sect.~\ref{metas}. Choosing the functions $N(Q)$ and $\alpha(Q)$ so that they satisfy the properties $(i)$-$(iii)$ in the version of Lemma~\ref{inv:prop}, we obtain the following corollary.
\begin{C1bis}
Let $(M,H,g_0)$ be a closed regular sub-Riemannian manifold. Then there exist positive constants $\bar C_0=\bar C_0(Q)$ and $\bar C_1=\bar C_1(Q)$ depending only on the Hausdorff dimension $Q$ of $M$ such that for any  sub-Riemannian metric $g$ conformal to $g_0$  the counting function $\tilde N_g(\lambda)$ for the sub-Laplacian $(-\tilde\Delta_b)$ satisfies the inequality
$$
\tilde N_g(\lambda)\geqslant \bar C_1(\mathcal S_g(M)^{2/Q}\lambda-\bar C_0(\mathit{Min}\mathcal S(M,[g_0]))^{2/Q})^{Q/2}
$$
for any $\lambda\geqslant \bar C_0(\mathit{Min}\mathcal S(M,[g_0])/\mathcal P_g(M))^{2/Q}$, where  $\mathit{Min}\mathcal S(M,[g_0])$ is the conformal minimal volume of $M$ in the sense of relation~\eqref{inv:haus}. 
\end{C1bis}
As a direct consequence, we obtain the following eigenvalue bounds
$$
\tilde\lambda_k(g)\mathcal{S}_g(M)^{2/Q}\leqslant \bar C_0\cdot(\mathit{Min}\mathcal S(M,[g_0]))^{2/Q}+\bar C_1^{-2/Q}\cdot k^{2/Q},
$$
which generalise bounds in~\eqref{ei:popp}, see  Corollary~\ref{coro1}, to arbitrary closed regular sub-Rieman\-nian manifolds. Finally, mention that Theorem~\ref{main:t2} also has a  version for the Hausdorff sub-Laplacian $(-\tilde\Delta_b)$.

\section{Eigenvalue bounds on contact manifolds}
\label{contact}
\subsection{Sasakian structures and the proof of Theorem~\ref{t3}}
Let $(M,\theta,\phi,g)$ be a contact metric manifold of dimension $(2\ell+1)$, see Example~\ref{cont:def} for the notation. Recall that it is called {\em Sasakian} if the following relation holds:
$$
[\phi,\phi](X,Y)+d\theta(X,Y)\xi=0
$$
for any vector fields $X$ and $Y$ on $M$. Above $\xi$ stands for the Reeb field, and $[\phi,\phi]$ is the Nijenhuis tensor
$$
[\phi,\phi](X,Y)=\phi^2[X,Y]+[\phi X,\phi Y]-\phi[\phi X,Y]-\phi[X,\phi Y].
$$
In dimension $3$ the Sasakian hypothesis is equivalent to the Reeb field $\xi$ being Killing. Sasakian manifolds are often viewed as odd-dimensional versions of K\"ahler manifolds. We refer to~\cite{Blair10} for the detailed discussion of their properties and examples. 

Recall that the {\em Tanaka-Webster connection} on a Sasakian manifold is a unique linear connection $\nabla$ such that $\theta,\phi$, and $g$ are parallel and whose torsion satisfies the relations
$$
T(X,Y)=d\theta(X,Y)\xi\quad\text{and}\quad T(\xi,X)=0
$$
for any horizontal vector fields $X$ and $Y$ on $M$. In particular, the Reeb field $\xi$ is also parallel, and the Ricci curvature tensor satisfies the relation 
$$
\mathit{Ricci}(X,\xi)=0\qquad\text{for any vector field }X.
$$
In other words, only the restriction of $\mathit{Ricci}$ to a contact distribution may carry non-trivial geometric information.

The proof of Theorem~\ref{t3} is based on the volume doubling properties of Sasakian manifolds with lower Ricci curvature bound, studied in the series of papers~\cite{BBG,BG,BBGM}. We summarise these results in the following proposition.
\begin{prop}
\label{cdi}
Let $(M,\theta,\phi,g)$ be a complete Sasakian manifold whose Ricci curvature is bounded below by $-1$. Then there exist positive constants $\bar C_1$ and $\bar C_2$ depending on the dimension of $M$ only such that
$$
\mathit{Vol}_g(B(x,2r))\leqslant \bar{C}_1\exp(\bar{C}_2r^2)\cdot\mathit{Vol}_g(B(x,r))
$$
for any $x\in M$ and $r>0$, where $B(x,r)$ and $B(x,2r)$ are Carnot-Caratheodory balls. Moreover, if the Ricci curvature is non-negative, then there exists a constant $\bar{C}$ such that
$$
\mathit{Vol}_g(B(x,2r))\leqslant \bar{C}\cdot\mathit{Vol}_g(B(x,r))
$$
for any $x\in M$ and $r>0$.
\end{prop}
The first statement of Prop.~\ref{cdi} follows by combination of~\cite[Theorem~1.7]{BG} and~\cite[Theorem~6]{BBGM}. The second can be derived from the proof of~\cite[Theorem~6]{BBGM}, see~~\cite[Remark~4]{BBGM}, and is independently proved in~\cite{BBG}.

\begin{proof}[The proof of Theorem~\ref{t3}]
By Prop.~\ref{cdi} we see that there exists a constant $\bar{C}_0$ that depends on a dimension  of $M$ only such that
$$
\mathit{Vol}_{g_0}(B(x,2r))\leqslant \bar{C}_0\cdot\mathit{Vol}_{g_0}(B(x,r))\qquad\text{for any }0<r\leqslant 2, ~x\in M.
$$
A standard argument, see for example the proof of Lemma~\ref{norm:c}, implies that the metric $g_0$ satisfies a {\em local covering property}: there exists a constant $N=N(\bar{C}_0)$ that depends on $\bar{C_0}$ only such that any ball $B(x,r)$ with $0<r\leqslant 1$ can be covered by $N$ balls of radius $r/2$.
By Example~\ref{cont:na}, the quantity $\hat{\mathcal P}_x(\hat{B}_x)$ equals the volume of a unit ball in the Heisenberg group, and we clearly have
$$
\mathit{Vol}_{g_0}(B(x,r))\leqslant\alpha_1(g_0)\hat{\mathcal P}_x(\hat{B}_x)r^{2\ell+2}\qquad\text{for any }0<r\leqslant 1, ~x\in M.
$$ 
Thus, we see that the metric $g_0$ is locally $(N,\alpha)$-normalised for $N=N(\bar{C}_0)$ and $\alpha=\alpha_1(g_0)$. By Theorem~\ref{main:t1} there exist constants $C_0=C_0(N)$ and $C_1=C_1(N,\ell)$ such that for any sub-Riemannian metric $g$ conformal to $g_0$ on the contact distribution the sub-Laplacian eigenvalues $\lambda_k(g)$ satisfy the following inequalities
$$
\lambda_k(g)\mathit{Vol}_g(M)^{1/(\ell+1)}\leqslant C_0\cdot(\mathit{Min\mathcal P}[g_0])^{1/(\ell+1)}+C_1\cdot(\alpha_1(g_0) \omega_\ell)^{1/(\ell+1)}k^{1/(\ell+1)}
$$
for any $k\geqslant 1$, where $\omega_\ell$ is the volume of a unit ball in the Heisenberg group $\mathbb H^\ell$. First, note that the constants $C_0$ and $C_1$ above now depend on the dimension of $M$ only. Second, by the definition of the conformal $(N,\alpha)$-minimal volume we have $\mathit{Min\mathcal P}[g_0]\leqslant\mathit{Vol_{g_0}}(M)$. Combining these observations, we obtain the eigenvalue bounds stated in the theorem.
\end{proof}

\subsection{Volume comparison theorems and their consequences}
Now we discuss recent volume comparison theorems due to Agrachev and Lee~\cite{AL} and Lee and Li~\cite{LL}, which give bounds for the volume growth modulus $\alpha(g)$ and allow to dispense with this quantity in the eigenvalue bounds.

Recall that the Heisenberg group $\mathbb H^\ell$ can be viewed as a product $\mathbb C^\ell\times\mathbb R$ with a Sasakian structure $(\theta,\phi,g)$, where
$$
\theta=dt+\frac{1}{2}\sum_i(y^idx^i-x^idy^i),\qquad (x^1+iy^1,\ldots,x^\ell+iy^\ell,t)\in\mathbb C^\ell\times\mathbb R,
$$
the frame
$$
X_i=\frac{\partial}{\partial x^i}-\frac{1}{2}y^i\frac{\partial}{\partial t},\qquad Y_i=\frac{\partial}{\partial y^i}+\frac{1}{2}x^i\frac{\partial}{\partial t},\qquad Z=\frac{\partial}{\partial t}
$$
is orthonormal in a Riemannian metric $g$, and $\phi$ satisfies the relations
$$
\phi(X_i)=Y_i,\quad\qquad \phi(Y_i)=-X_i,\quad\qquad\phi(Z)=0.
$$
A straightforward calculation shows, see for example~\cite{LL}, that the curvature tensor of a Tanaka-Webster connection on $\mathbb H^\ell$ vanishes. The following result due to~\cite{LL} uses $\mathbb H^\ell$ as a comparison space to bound the volumes of Carnot-Caratheodory balls.
\begin{prop}
\label{vc1}
Let $(M,\theta_0,\phi_0,g_0)$ be a complete Sasakian manifold of dimension $(2\ell+1)$ whose horizontal sectional curvatures of a Tanaka-Webster connection are non-negative. Then for any $x\in M$ the volume $\mathit{Vol}_{g_0}(B(x,r))$ of a Carnot-Caratheodory ball $B(x,r)$ is not greater than the volume of a ball of the same radius in the Hesenberg group, that is
$$
\mathit{Vol}_{g_0}(B(x,r))\leqslant\omega_\ell r^{2\ell+2}\qquad\text{for any}\quad r>0,
$$
where $\omega_\ell$ is the volume of a unit ball in $\mathbb H^\ell$.
\end{prop}
Recall that the {\em volume growth} $\alpha_0(g)$ is defined as the quantity
$$
\alpha_0(g_0)=\sup\{\mathit{Vol}_{g_0}(B(x,r))/(\omega_\ell r^{2\ell+2}): x\in M, r>0\}.
$$
In particular, Prop.~\ref{vc1} implies that for a  Sasakian manifold $M$ whose horizontal sectional curvatures are non-negative, we have  $\alpha_0(g_0)\leqslant 1$. Combing this fact with Corollary~\ref{coro2}, we see that on a compact manifold $M$ for any contact metric structure $(\theta,\phi,g)$ with $\theta=e^\varphi\theta_0$ and $j(\phi)=j(\phi_0)$ the sub-Laplacian eigenvalues $\lambda_k(g)$ satisfy the inequalities
\begin{equation}
\label{kor1}
\lambda_k(g)\mathit{Vol}_g(M)^{1/(\ell+1)}\leqslant C\cdot k^{1/(\ell+1)},
\end{equation}
where the constant $C$ depends on the dimension of $M$ only. Using Theorem~\ref{main:t2} we are able to obtain a more general result for not necessarily compact manifolds.
\begin{theorem}
\label{main:t5}
Let $(M,\theta_0,\phi_0,g_0)$ be a complete Sasakian manifold of dimension $(2\ell+1)$ whose horizontal sectional curvatures of a Tanaka-Webster connection are non-negative. Then there exists a constant $C>0$ depending on the dimension of $M$ only such that for any subdomain $\Omega\subset M$ and any contact metric structure $(\theta,\phi,g)$ on $\Omega$ with $\theta=e^\varphi\theta_0$, $j(\phi)=j(\phi_0)$, and $\mathit{Vol}_g(\Omega)<+\infty$ the counting function $N_g(\Omega,\lambda)$ of the Neumann sub-Laplacian  $(-\Delta_b)$ on $\Omega$ satisfies the inequality
$$
N_g(\Omega,\lambda)\geqslant \lfloor C\cdot\mathit{Vol}_g(\Omega)\lambda^{\ell+1}\rfloor\qquad\text{for any~}\lambda\geqslant 0.
$$
Besides, if a subdomain $\Omega\subset M$ is compact, then the floor function is unnecessary in the estimate above.
\end{theorem}
\begin{proof}
By Prop.~\ref{cdi} we see that there exists a universal constant $\bar{C}_0$ such that
$$
\mathit{Vol}_{g_0}(B(x,2r))\leqslant \bar{C}_0\cdot\mathit{Vol}_{g_0}(B(x,r))\qquad\text{for any }r>0, ~x\in M.
$$
By a standard argument, see for example the proof of Lemma~\ref{norm:c}, we see that the metric $g_0$ satisfies a {\em global covering property}: there exists a constant $N_0=N(\bar{C}_0)$ such that any ball $B(x,r)$ with $r>0$ can be covered by $N_0$ balls of radius $r/2$. Combining this property with Prop.~\ref{vc1}, we conclude that the metric $g_0$ is globally $(N_0,1)$-normalised. Thus, by Theorem~\ref{main:t2} there exists a constant $C=C(N_0,\ell)$ such that for any $\Omega\subset M$ and any sub-Riemannian metric $g$ on $\Omega$ that is conformal to $g_0$ on a contact distribution and has finite volume, the counting function $N_g(\Omega,\lambda)$ satisfies the inequality
$$
N_g(\Omega,\lambda)\geqslant \lfloor C\cdot (\max\hat{\mathcal P}_g(\hat{B}_x))^{-1}\mathcal P_g(\Omega)\lambda^{\ell+1}\rfloor
$$
for any $\lambda\geqslant 0$. Now if a metric $g$ is associated with a contact metric structure, then by Example~\ref{cont:na} the value $\hat{\mathcal P}_g(\hat{B}_x)$ equals $\omega_\ell$ (the volume of a unit ball in the Heisenberg group $\mathbb H^\ell$), and the volume $\mathcal P_g(\Omega)$ coincides with the Riemannian volume $\mathit{Vol}_g(\Omega)$.
\end{proof}
Mention that the eigenvalue bounds~\eqref{kor1} can be viewed as  a version of the Korevaar result~\cite{Korv} on the Laplace eigenvalue bounds for Riemannian metrics conformal to a metric of non-negative Ricci curvature. In view of this analogy we pose the following question.
\begin{quest}
Does the conclusion of Theorem~\ref{main:t5} hold under the hypotheses that the Ricci curvature of a Tanaka-Webster connection of $g_0$ is non-negative?
\end{quest}

We proceed with a discussion of volume comparison results for $3$-dimensional Sasakian manifolds. Recall that the corresponding Sasakian space forms are $3$-dimensinal Lie groups $G_\kappa$ whose Lie algebra has a basis $\{X,Y,Z\}$ that satisfies the relations
$$
[X,Y]=Z,\qquad [X,Z]=-\kappa Y,\qquad [Y,Z]=\kappa X.
$$
The parameter $\kappa$ here takes the values $1$, $0$, and $-1$, which correspond to the cases when $G_\kappa$ is $SU(2)$, $\mathbb H^1$, and $SL(2)$ respectively. The left-invariant contact form $\theta$ such that
$$
\theta(X)=\theta(Y)=0\qquad\text{ and }\qquad \theta(Z)=-1,
$$
the metric $g$ that makes $\{X,Y,Z\}$ orthonormal, and the endomorphism
$$
\phi(X)=-Y,\qquad\phi(Y)=X,\qquad\phi(Z)=0,
$$
form a Sasakian structure on $G_\kappa$. As is known~\cite{AL}, the Tanaka Webster sectional curvature of a plane spanned by $X$ and $Y$ on $G_\kappa$ equals $\kappa$. The space forms for all real values of $\kappa$ can be obtained from the examples above by an appropriate scaling.

In~\cite{AL} Agrachev and Lee prove the following volume comparison theorem for Carnot-Caratheodory balls.
\begin{prop}
\label{vc2}
Let $(M,\theta_0,\phi_0,g_0)$ be a $3$-dimensional complete Sasakian manifold whose Ricci curvature of a Tanaka-Webster connection is bounded below by $\kappa$, that is
$$
\mathit{Ricci}(X,X)\geqslant\kappa\cdot g_0(X,X)\qquad\text{for any horizontal vector field }X.
$$
Then for any $x\in M$ the volume of a Carnot-Caratheodory ball $B(x,r)$ satisfies the inequality
$$
\mathit{Vol}_{g_0}(B(x,r))\leqslant\omega_\kappa(r)
$$
for any $r>0$ if $\kappa\geqslant 0$, and any $0<r\leqslant 2\sqrt{2}\pi/\sqrt{-\kappa}$ if $\kappa<0$, where $\omega_\kappa(r)$ stands for the volume of a Carnot-Caratheodory ball of radius $r$ in a Sasakian space form of constant curvature $\kappa$.
\end{prop}
The restriction on the radius in the negative curvature case $\kappa<0$ is related to the fact that $SL(2)$ is not simply connected. In~\cite{AL} the authors give an explicit formula for the quantity $\omega_\kappa(r)$ for the above values of $r$ as the integral 
$$
\omega_\kappa(r)=2\pi\int_{\Omega_\kappa(r)}b_\kappa(t,z)tdtdz,
$$
where $\Omega_\kappa(r)=\{(t,z)\in\mathbb R^2: t\in (0,r], \kappa t^2+z^2\leqslant 4\pi^2\}$, and the function $b_\kappa$ is defined as
$$
b_\kappa(t,z)=\left\{
\begin{array}{ll}
t^2(2-2\cos\tau-\tau\sin\tau)/\tau^4 & \text{if }\quad \kappa t^2+z^2>0\\
t^2/12& \text{if }\quad \kappa t^2+z^2=0\\
t^2(2-2\cosh\tau+\tau\sinh\tau)/\tau^4 & \text{if }\quad \kappa t^2+z^2<0
\end{array}
\right.
$$
with $\tau=\sqrt{\abs{\kappa t^2+z^2}}$. In particular, it is straightforward to see that there exists a universal constant $\omega_*$ such that $\omega_\kappa(r)\leqslant\omega_*r^4$ for any $r>0$ if $\kappa\geqslant 0$, and any $0<r<2\sqrt{2}\pi/\sqrt{-\kappa}$ if $\kappa<0$. 

The next result shows that in dimension $3$ the volume growth function in Theorem~\ref{t3} can be dispensed with.
\begin{theorem}
\label{main:t6}
Let $(M,\theta_0,\phi_0,g_0)$ be a complete $3$-dimensional Sasakian manifold whose Ricci curvature of a Tanaka-Webster connection is bounded below by $-1$. Then there exist universal positive constants $C_0$ and $C_1$ such that for any subdomain $\Omega\subset M$ and any contact metric structure $(\theta,\phi,g)$ on $\Omega$ with $\theta=e^\varphi\theta_0$, $j(\phi)=j(\phi_0)$, and $\mathit{Vol}_g(\Omega)<+\infty$ the counting function $N_g(\Omega,\lambda)$ of the Neumann sub-Laplacian  $(-\Delta_b)$ on $\Omega$ satisfies the inequality
$$
N_g(\Omega,\lambda)\geqslant \lfloor C_1\mathit{Vol}_g(\Omega)\left(\lambda-C_0\frac{\mathit{Vol}_{g_0}(\Omega)^{1/2}}{\mathit{Vol}_{g}(\Omega)^{1/2}}\right)^2\rfloor
$$
for any $\lambda\geqslant C_0(\mathit{Vol}_{g_0}(\Omega)/\mathit{Vol}_{g}(\Omega))^{1/2}$. Besides, if a subdomain $\Omega\subset M$ is compact, then the floor function is unnecessary in the estimate above.
\end{theorem}
\begin{proof}
Following the argument in the proof of Theorem~\ref{t3}, by Prop.~\ref{cdi} we see that the metric $g_0$ satisfies the local covering property: there exists a universal constant $N_0$ such that any ball $B(x,r)$ with $0<r\leqslant 1$ can be covered by $N_0$ balls of radius $r/2$. By Prop.~\ref{vc2} together with the inequality $\omega_{-1}(r)\leqslant \omega_* r^4$, we conclude that the metric $g_0$ is $(N_0,\alpha_0)$-normalised with $\alpha_0=\omega_*/\omega_1$, where $\omega_1$ is the volume of a unit ball in $\mathbb H^1$, see Example~\ref{cont:na}. Now the statement follows directly from Theorem~\ref{main:t1}.
\end{proof}

As a direct consequence of Theorem~\ref{main:t6} we obtain the following statement.
\begin{cor}
\label{main:c7}
Let $(M,\theta,\phi,g)$ be a complete $3$-dimensional Sasakian manifold whose Ricci curvature of a Tanaka-Webster connection is bounded below, $\mathit{Ricci}\geqslant -a^2$. Then there exist universal positive constants $C_0$ and $C_1$ such that for any subdomain $\Omega\subset M$ of finite volume, $\mathit{Vol}_g<+\infty$, the counting function $N_g(\Omega,\lambda)$ of the Neumann sub-Laplacian $(-\Delta_b)$ on $\Omega$ satisfies the inequality
$$
N_g(\Omega,\lambda)\geqslant \lfloor C_1\mathit{Vol}_g(\Omega)(\lambda-C_0a^2)^2\rfloor
$$
for any $\lambda\geqslant C_0a^2$. Besides, if a subdomain $\Omega\subset M$ is compact, then the floor function is unnecessary in the estimate above.
\end{cor}
When a subdomain $\Omega\subset M$ above is compact, using the relation 
$$
\lambda_k(g)=\inf\left\{\lambda\geqslant 0: N_g(\lambda)\geqslant k\right\},
$$
it is straightforward to conclude that the lower bound on the counting function $N_g(\Omega,\lambda)$ in Corollary~\ref{main:c7} implies that
$$
\lambda_k(\Omega,g)\leqslant C_0a^2+C_1^{-1/2}\cdot\left(\frac{k}{\mathit{Vol}_g(\Omega)}\right)^{1/2}
$$
for any $k=1,2,\ldots.$ This inequality is a version for $3$-dimensional Sasakian manifolds of a classical Buser's inequality for Laplace eigenvalues on Riemannian manifolds, see~\cite{Bu79,CM08}.
\begin{quest}
Does the conclusion of Theorem~\ref{main:t6} hold in arbitrary dimension?
\end{quest}
As a partial answer to the question above, mention that it is likely that the volume comparison theorems obtained in~\cite{LL} can be extended to Sasakian manifolds, whose horizontal sectional curvatures are bounded below by a negative constant. Given such a result, the argument in the proof of Theorem~\ref{main:t6} yields the lower bound for the counting function under the hypothesis that the horizontal sectional curvatures of $g_0$ are bounded below by $-1$. However, at the moment of writing it is unclear whether a similar volume comparison theorem holds under the lower bound for the Ricci curvature.


\section{Proofs of Theorems~\ref{main:t1} and~\ref{main:t2}}
\label{metas}
\subsection{Decompositions of metric measure spaces}
The proofs of Theorems~\ref{main:t1} and~\ref{main:t2} are based on the constructions of disjoint subsets in metric measure spaces carrying a sufficient amount of mass. Below $(X,d)$ denotes a separable metric space. We start with recalling the following definition.
\begin{defn}
For an integer $N\geqslant 1$ a metric space $(X,d)$ is said to satisfy the {\em local} $N$-{\em covering property} if each metric ball of radius $0<r\leqslant 1$ can be covered by $N$ balls of radius $r/2$. If each metric ball of any radius $r>0$ can be covered by $N$ balls of radius $r/2$, then $(X,d)$ is said to satisfy the {\em global} $N$-{\em covering property.}
\end{defn}
Building on the ideas of Korevaar~\cite{Korv}, Grigoryan, Netrusov, and Yau~\cite{GNY} showed that on certain metric spaces with global covering properties for any non-atomic finite measure one can always find a collection of disjoint annuli carrying a controlled amount of measure. Below by an annulus $A$ in $(X,d)$ we mean a subset of the following form
$$
\{x\in X: r\leqslant d(x,a)<R\},
$$
where $a\in X$ and $0\leqslant r<R<+\infty$. The real numbers $r$ and $R$ above are called the {\em inner} and the {\em outer radii} respectively; the point $a$ is the {\em centre} of an annulus $A$. We also use the notation $2A$ for the annulus
$$
\{x\in X: r/2\leqslant d(x,a)<2R\}.
$$
The following statement is due to~\cite[Corollary~3.12]{GNY}, see also~\cite{GY99}.
\begin{prop}
\label{prop:glo}
Let $(X,d)$ be a separable metric space whose all metric balls are precompact. Suppose that it satisfies the global $N$-covering property for some $N\geqslant 1$. Then for any finite non-atomic measure $\mu$ on $X$ and any positive integer $k$ there exists a collection of $k$ disjoint annuli $\{2A_i\}$ such that
$$
\mu(A_i)\geqslant c\mu(X)/k\qquad\text{for any}\quad 1\leqslant i\leqslant k,
$$
where $c$ is a positive constant that depends on $N$ only. 
\end{prop}
Mention that when $(X,d)$ is a locally compact length space (for example, a sub-Rieman\-nian manifold with a Carnot-Caratheodory metric), the hypothesis that all metric balls are precompact is equivalent  to the completeness of $(X,d)$, see~\cite{BBI}. The following statement, due to~\cite[Theorem~2.1]{Ha11}, is a version of Prop.~\ref{prop:glo} for metric spaces with local covering properties.
\begin{prop}
\label{prop:loc}
Let $(X,d)$ be a separable metric space whose all metric balls are precompact. Suppose that it satisfies the local $N$-covering property for some $N\geqslant 1$. Then for any finite non-atomic measure $\mu$ on $X$ and any positive integer $k$ there exists a collection of $k$ subsets $\{A_i\}$ such that
$$
\mu(A_i)\geqslant c\mu(X)/k\qquad\text{for any}\quad 1\leqslant i\leqslant k,
$$
where $c$ is a positive constant that depends on $N$ only, and
\begin{itemize}
\item[-] either all $A_i$ are annuli such that  $2A_i$ are mutually disjoint, and the outer radii of the latter are not greater than  $1$;
\item[-] or all $A_i$ are precompact subdomains whose $\rho$-neighbourhoods  
$$
A_i^\rho=\{x\in X:\dist(x,A_i)\leqslant \rho\}
$$
are mutually disjoint, where $\rho=(1600)^{-1}$.
\end{itemize}
\end{prop}
The proof of Prop.~\ref{prop:loc} is based on the combination of the method developed by Grygoryan, Netrusov, and Yau~\cite{GNY} together with the ideas by Colbois and Marten~\cite{CM08}, see~\cite{Ha11} for details.

\subsection{Lipschitz functions on Carnot-Caratheodory spaces}
Let $(M,H,g)$ be a regular sub-Riemannian manifold, and $(M,d_g)$ be the corresponding Carnot-Caratheo\-dory space. For a compact subdomain $D\subset M$ and a real number $1\leqslant p<+\infty$ consider the sub-Rieman\- nian Sobolev norm
\begin{equation}
\label{norm}
\left(\int_D\abs{u}^pd\mathcal{P}_g\right)^{1/p}+\left(\int_D\abs{\nabla_bu}^pd\mathcal{P}_g\right)^{1/p}
\end{equation}
Recall that a real-valued function $u$ on $M$ is called Lipschitz in the sense of the Carnot-Caratheodory metric $d$, if there exists a constant $L$  such that
$$
\abs{u(x)-u(y)}\leqslant L\cdot d_g(x,y)\qquad\text{ for any }x\text{ and }y\text{ in }M.
$$
As is known, any smooth function is locally Lipschitz in the above sense. We proceed with the following proposition, due to~\cite{FSSC96,FSSC97} and~\cite{GN98}.
\begin{prop}
\label{lip}
Let $D$ be a compact subdomain in  a regular sub-Riemannian manifold $(M,H,g)$, and $u$ be a Lipschitz function in the sense of the Carnot-Caratheodory distance $d_g$ on it.
\begin{itemize}
\item[(i)] Then the distributional derivative $\nabla_bu$ is a measurable bounded vector-field whose norm $\abs{\nabla_bu}$ is not greater than the Lipschitz constant $L$. 
\item[(ii)] The function $u$ can be approximated in the norm~\eqref{norm} by smooth functions on $D$. Moreover, if $u$ is compactly supported, that it can be approximated in the norm~\eqref{norm} by smooth compactly supported functions.
\end{itemize}
\end{prop}
Mention that the second statement of Prop.~\ref{lip} is an application of the mollification technique, explained in~\cite[p.79-81]{GN98}. Prop.~\ref{lip} shows that compactly supported Lipschitz functions can be used as test-functions for the counting function $N_g(\Omega,\lambda)$ of the Neumann eigenvalue problem on $\Omega\subset M$. Below we construct such functions out of the distance function $x\mapsto\dist(x,A)$ to a compact subset $A$.

\subsection{Proof of Theorem~\ref{main:t2}}
Let $(M,H,g_0)$ be a complete regular sub-Riemannian manifold that is globally $(N,\alpha)$-normalised, and $\Omega\subset M$ be a subdomain with a smooth boundary. We consider $M$ as a metric space with the Carnot-Caratheodory metric $d_{g_0}$ and with a measure $\mu=\left.\mathcal P_g\right|\Omega$, where $g$ is a metric conformal to $g_0$ on $\Omega$. Let $C_*=C_*(N)$ be a constant from Prop.~\ref{prop:glo} applied to $(M,d_{g_0})$. For a given $\lambda>0$ denote by $k$ the integer
$$
\lfloor C\cdot(\alpha\max_\Omega\hat{\mathcal P}_x(\hat{B}_x))^{-1}\mu(M)\lambda^{Q/2}\rfloor,
$$
where $C=(C_*/48)^{Q/2}$. For a proof of the theorem it is sufficient to construct $k$ linearly independent Lipschitz test-functions $u_i$ such that
\begin{equation}
\label{t2:eq1}
\int\limits_\Omega\abs{\nabla_bu_i}^2d\mathcal P_g<\lambda\int\limits_\Omega u_i^2d\mathcal P_g.
\end{equation}
By Prop.~\ref{prop:glo} there exists a collection $\{A_i\}$ of $3k$ annuli in $(M,d_{g_0})$ such that
\begin{equation}
\label{t2:eq2}
\mu(A_i)\geqslant C_*\mu(M)/(3k)\qquad\text{for any}\quad i=1,\ldots,3k,
\end{equation}
and the annuli $\{2A_i\}$ are disjoint. The latter implies that
$$
\sum_{i=1}^{3k}\mu(2A_i)\leqslant\mu(M),
$$
and hence, there exist at least $k$ sets $2A_i$ such that
\begin{equation}
\label{t2:eq3}
\mu(2A_i)\leqslant \mu(M)/k.
\end{equation}
Without loss of generality, we may suppose that these inequalities hold for $i=1,\ldots,k$. For such an $i$ denote by $a_i$, $r_i$ and $R_i$ the centre, the inner radius and the outer radius of $A_i$ respectively. We define the Lipschitz test-functions $u_i$ on $2A_i$ by the following relation
$$
u_i(x)=\left\{
\begin{array}{ll}
1 & \text{for }\quad r_i\leqslant d_{g_0}(x,a_i)<R_i;\\
2d_{g_0}(x,a_i)/r_i-1 & \text{for }\quad r_i/2\leqslant d_{g_0}(x,a_i)<r_i;\\
2-d_{g_0}(x,a_i)/R_i & \text{for }\quad R_i\leqslant d_{g_0}(x,a_i)<2R_i;\\
0 & \text{for }\quad d_{g_0}(x,a_i)\geqslant 2R_i\text{ ~or~ }d_{g_0}(x,a_i)<r_i/2.
\end{array}
\right.
$$ 
By Prop.~\ref{lip}, we see that the horizontal gradient satisfies 
$$
\abs{\nabla_b u_i}(x)\leqslant\left\{
\begin{array}{ll}
2/r_i & \text{when}\quad r_i/2\leqslant d_{g_0}(x,a_i)<r_i,\\
1/R_i & \text{when}\quad R_i\leqslant d_{g_0}(x,a_i)<2R_i,\\
\end{array}
\right.
$$ 
and vanishes at all other points in $M$. Note that if $Q$ is the Hausdorff dimension of $M$, then the integral $\int\abs{\nabla_bu}^Qd\mathcal P_g$ is invariant under the conformal change of a metric; this is a consequence of relation~\eqref{conformal} in Section~\ref{prem}. Thus, we obtain
\begin{multline*}
\int_\Omega\abs{\nabla_bu_i}^Qd\mathcal P_g=\int_\Omega\abs{\nabla_bu_i}^Qd\mathcal P_{g_0}\leqslant \left(\frac{2}{r_i}\right)^Q\mathcal P_{g_0}(B(a_i,r_i))+\left(\frac{1}{R_i}\right)^Q\mathcal P_{g_0}(B(a_i,2R_i))\\
\leqslant 2\cdot 2^Q(\alpha\max_\Omega\hat{\mathcal P}_x(\hat{B}_x))< 4^Q(\alpha\max_\Omega\hat{\mathcal P}_x(\hat{B}_x)),
\end{multline*}
where in the second inequality we used the growth property of the globally $(N,\alpha)$-normalised metric $g_0$. Now by the H\"older inequality, we have
\begin{multline*}
\int_\Omega\abs{\nabla_bu_i}^2d\mathcal P_g\leqslant \left(\int_\Omega\abs{\nabla_bu_i}^Qd\mathcal P_g\right)^{2/Q}\left(\int_{2A_i\cap\Omega}1d\mathcal P_g\right)^{1-2/Q}\\
< 16\cdot (\alpha\max_\Omega\hat{\mathcal P}_x(\hat{B}_x))^{2/Q}\mu(2A_i)^{1-2/Q}.
\end{multline*}
Finally, using relations~\eqref{t2:eq2} and~\eqref{t2:eq3}, we obtain
$$
\left(\int_\Omega\abs{\nabla_bu_i}^2d\mathcal P_g\right)/\left(\int_\Omega u_i^2d\mathcal P_g\right)<
\frac{16}{C_*}(\alpha\max_\Omega\hat{\mathcal P}_x(\hat{B}_x))^{2/Q}\frac{\mu(M)^{1-2/Q}(3k)}{k^{1-2/Q}\mu(M)}\leqslant\lambda,
$$
where in the last inequality we used the definition of $k$. Thus, the first statement of the theorem is demonstrated. To prove the estimate for the counting function when a subdomain $\Omega\subset M$ is compact, we note that in addition to the functions $u_i$ above, any non-zero constant function qualifies as an extra test-function. \qed

\subsection{Proof of Theorem~\ref{main:t1}}
Let $\Omega\subset M$ be a subdomain with a smooth boundary. Pick a  locally $(N,\alpha)$-normalised metric $g_*$ from a given conformal class on $M$, and denote by $C_*=C_*(N)$ the constant from Prop.~\ref{prop:glo} applied to the Carnot-Caratheodory space $(M,d_{g_*})$. We define new constants
$$
C_0(N)=3\cdot(1600)^2/C_*\qquad\text{and}\qquad C_1(N,Q)=(C_*/48)^{Q/2}.
$$
Let $g$ be a metric on $\Omega$ conformal to $g_*$, and $\lambda> C_0(\mathit{Min}\mathcal P(\Omega,[g_0])/\mathcal P_g(\Omega))^{2/Q}$ be a given real number. Since $C_0=C_0(N)$ depends only on $N$, we may choose another locally $(N,\alpha)$-normalised metric conformal to $g_*$, which we denote by $g_0$, such that 
\begin{equation}
\label{aux:lambda}
\lambda> C_0(\mathcal P_{g_0}(\Omega)/\mathcal P_g(\Omega))^{2/Q}.
\end{equation} 
Now the strategy is to apply Prop.~\ref{prop:loc} to the Carnot-Caratheodory space $(M,d_{g_0})$ equipped with a measure $\mu=\left.\mathcal P_g\right|\Omega$. First, denote by $\mu_0$ the measure $\left.\mathcal P_{g_0}\right|\Omega$, and by $k$ the integer
$$
\lfloor C_1(\alpha\max_\Omega\hat{\mathcal P}_x(\hat{B}_x))^{-1}(\mu(M)^{2/Q}\lambda-C_0(\mu_0(M))^{2/Q})^{Q/2}\rfloor.
$$
By Prop.~\ref{prop:loc} there exists a collection $\{A_i\}$ of $3k$ subsets in $(M,d_{g_0})$ such that
\begin{equation}
\label{t1:eq1}
\mu(A_i)\geqslant C_*\mu(M)/(3k)\qquad\text{for any}\quad i=1,\ldots,3k,
\end{equation}
and either
\begin{itemize}
\item[(i)] all $A_i$ are annuli, and $2A_i$ are mutually disjoint and the outer radii of the latter are not greater than $1$, or
\item[(ii)] all $A_i$ are precompact subdomains whose $\rho$-neighbourhoods $A_i^\rho$ are mutually disjoint, where $\rho=(1600)^{-1}$.
\end{itemize}
In the first case we proceed following the lines of the proof of Theorem~\ref{main:t2} to construct $k$ linearly independent test-functions $u_i$ such that
$$
\int\limits_\Omega\abs{\nabla_bu_i}^2d\mathcal P_g<\lambda\int\limits_\Omega u_i^2d\mathcal P_g.
$$
Now we explain the argument for the case~$(ii)$. First, since $A_i^{\rho}$ are mutually disjoint, we have
$$
\sum_{i=1}^{3k}\mu(A_i^\rho)\leqslant\mu(M)\qquad\text{and}\qquad\sum_{i=1}^{3k}\mu_0(A_i^\rho)\leqslant\mu_0(M).
$$
Hence, there exist at least $k$ subsets $A_i$ such that
\begin{equation}
\label{t1:eq2}
\mu(A_i^\rho)\leqslant \mu(M)/k\qquad\text{and}\qquad \mu_0(A_i^\rho)\leqslant \mu_0(M)/k.
\end{equation}
Without loss of generality, we may suppose that these inequalities hold for $i=1,\ldots,k$. For such an $i$  we define the Lipschitz test-functions $u_i$ on $A_i^\rho$ by the following relation
$$
u_i(x)=\left\{
\begin{array}{ll}
1 & \text{for }\quad x\in A_i\\
1-\dist(x,A_i)/\rho & \text{for }\quad x\in A_i^\rho\backslash A_i\\
0 & \text{for }\quad x\in M\backslash A_i^\rho.
\end{array}
\right.
$$ 
By Prop.~\ref{lip}, we see that $\abs{\nabla_bu_i}\leqslant 1/\rho$, and hence, we obtain
$$
\int_\Omega\abs{\nabla_bu_i}^Qd\mathcal P_g=\int_\Omega\abs{\nabla_bu_i}^Qd\mathcal P_{g_0}\leqslant 
\rho^{-Q}\mu_0(A_i^\rho)\leqslant \rho^{-Q}\mu_0(M)/k,
$$
where the first relation is the conformal invariance property, and in the last we used the second relation in~\eqref{t1:eq2}. By the H\"older inequality and the first relation in~\eqref{t1:eq2}, we further have
\begin{multline*}
\int_\Omega\abs{\nabla_bu_i}^2d\mathcal P_g\leqslant \left(\int_\Omega\abs{\nabla_bu_i}^Qd\mathcal P_g\right)^{2/Q}\left(\int_{A_i^\rho\cap\Omega}1d\mathcal P_g\right)^{1-2/Q}\\
\leqslant \rho^{-2}(\mu_0(M)/k)^{2/Q}\mu(A_i^\rho)^{1-2/Q}\\ 
\leqslant \rho^{-2}(\mu_0(M)/k)^{2/Q}(\mu(M)/k)^{1-2/Q}.
\end{multline*}
Now we use relation~\eqref{t1:eq1} to obtain
$$
\left(\int_\Omega\abs{\nabla_bu_i}^2d\mathcal P_g\right)/\left(\int_\Omega u_i^2d\mathcal P_g\right)\leqslant
3\frac{\rho^{-2}}{C_*}(\mu_0(M)/\mu(M))^{2/Q}<\lambda,
$$
where in the last inequality we used hypothesis~\eqref{aux:lambda}. Thus, the first statement of the theorem is proved. The estimate for the counting function on a compact subdomain $\Omega\subset M$ follows from the observation that any non-zero constant function qualifies as an extra test-function. \qed
\subsection{Proof of Theorem~\ref{main:t1}bis} The proof of the theorem follows the line of argument in the proof of Theorem~\ref{main:t1} with an obvious substitution of the Hausdorff measure $\mathcal S_g$ for the Popp measure $\mathcal P_g$. \qed

\appendix
\section{Proof of Lemma~\ref{volume}}
\label{append}
\noindent
Recall that the Lie algebra $\mathfrak g$ of the Carnot group $G$ has the grading
$$
\mathfrak g=\mathfrak g_1\oplus\mathfrak g_2,\quad\mathfrak g_2=[\mathfrak g_1,\mathfrak g_1],\quad [\mathfrak g_1,\mathfrak g_2]=0,\quad\dim\mathfrak g_2=1,
$$
where $\mathfrak g_1=H_e$ is a subspace in the Lie algebra corresponding to the fiber of the left-invariant distribution $H$ over the identity. Let $g$ be a left-invariant metric on $G$. First, we treat the case when the dimension of $\mathfrak g_1$ is even, $\dim\mathfrak g_1=2\ell$. Then there is an orthonormal basis $X_1,\ldots,X_\ell$, $Y_1,\ldots,Y_\ell$ of $\mathfrak g_1$ and a vector $Z\ne 0$ from $\mathfrak g_2$ such that
\[\begin{array}{ll}
[X_i,Y_i]=-b_iZ, & i=1,\ldots,\ell, \\ 
{[X_i,Y_j]=0}, & i\ne j, \\
{[X_i,Z]=[Y_i,Z]=0}, &  i=1,\ldots,\ell,
\end{array}\]
where $b_i\geqslant 0$. It is worth mentioning that by the Campbell-Hausdorff formula these relations recover the product structure on $G$. More precisely, denoting points $q\in G$ by triples $(x,y,z)$, where $x,y\in\mathbb R^\ell$ and $z\in\mathbb R$, one can write the group law in these coordinates
$$
q\cdot q'=(x+x',y+y',z+z'-\frac{1}{2}\sum_{i=1}^\ell b_i(x_ix'_i-y_iy'_i)).
$$
Without loss of generality, we may assume that the group $G$ does not have a Euclidean factor, that is all the $b_i$'s are strictly positive. Further, since the scaling of a metric $g$ does not change the volume of a unit ball, we may assume that $\max b_i$ equals $1$. In~\cite[Sect.~5]{ABB12} the authors compute the volume of a unit ball in $G$. Under our assumptions, it is given by the following formula
\begin{equation}
\label{voub}
\mathcal P_g(B_1)=\frac{C_\ell}{B^2}\int\limits_0^\pi s^{-(2\ell+2)}\sum_{i=1}^\ell\left(\prod_{j\ne i}\sin^2(b_js)\right)\sin(b_is)(b_is\cos(b_is)-\sin(b_is))ds,
\end{equation} 
where $B=\prod b_i$, and $C_\ell$ is the constant that depends only on $\ell$. This formula is obtained by introducing coordinates where the Popp measure coincides with the Lebesgue measure.

Now we show that the integral above is bounded independently of the values of the $b_i$'s. Using the Taylor series and the relations $0<b_i\leqslant 1$, we obtain
$$
\abs{b_is\cos(b_is)-\sin(b_is)}\leqslant c_*(b_is)^3\qquad\text{for any~}s\in [0,\pi],
$$
where $c_*$ is a constant that can be taken to be $\exp(\pi)$. Thus, for any $i=1,\ldots,\ell$, we have
\begin{multline*}
\frac{1}{B^2}\int\limits_0^\pi s^{-(2\ell+2)}\left(\prod_{j\ne i}\sin^2(b_js)\right)\sin(b_is)(b_is\cos(b_is)-\sin(b_is))ds\\ \leqslant\frac{c_*}{b_i^2}\int\limits_0^\pi s^{-4}\sin(b_is)(b_is)^3ds\leqslant c_*\pi,
\end{multline*}
where in the last inequality we again used $\abs{b_i}\leqslant 1$. Combining the relation above with formula~\eqref{voub}, we obtain a bound for the volume $\mathcal P_g(B_1)$ of a unit ball.

For a proof of the lemma it remains to consider the case when the dimension of $\mathfrak g_1$ is odd, $\dim\mathfrak g_1=2\ell+1$. Then one can choose an orthonormal basis $X_1,\ldots,X_\ell$, $Y_1,\ldots,Y_\ell$, $T$ of $\mathfrak g_1$ and a vector $Z\ne 0$ from $\mathfrak g_2$ such that
$$
[X_i,Y_i]=-b_iZ,\qquad i=1,\ldots,\ell,
$$
and all other brackets between the $X_i$'s, $Y_i$'s, $T$, and $Z$ vanish. As is explained in~\cite[Sect.~5.4]{ABB12}, the volume of a unit ball in $G$ in this case is given by the same formula~\eqref{voub}, and hence, is also bounded independently of a left-invariant metric $g$.

\subsection*{Acknowledgements} The authors are grateful to Andrei Agrachev, Davide Barilari, Bruno Colbois, and Ahmad El Soufi for the discussions on the subject. During the work on the paper the first author was partially supported by the Bavarian Equal Opportunities Sponsorship Programme (Bayerische Gleichstellungsf\"orderung) for postdoctoral researchers at the LMU Munich. The part of the project has been completed while the authors were visiting the Centre de recherches math\'ematiques in Montr\'eal; in particular, the first author acknowledges the support of the CRM-ISM postdoctoral fellowship.

\end{document}